\documentclass[a4paper,12pt,reqno]{amsart}
\usepackage{amsfonts}
\usepackage{amsmath}
\usepackage{amssymb}
\usepackage[a4paper]{geometry}
\usepackage{mathrsfs}
\usepackage{hyperref}
\renewcommand\eqref[1]{(\ref{#1})} 
\numberwithin{equation}{section}
\theoremstyle{plain}
\newtheorem{thm}{Theorem}[section]
\newtheorem{prop}[thm]{Proposition}
\newtheorem{cor}[thm]{Corollary}
\newtheorem{lem}[thm]{Lemma}

\theoremstyle{definition}

\newtheorem{rem}[thm]{Remark}

\renewcommand{\wp}{\mathfrak S}
\newcommand{\Rn}{\mathbb R^{n}}

\begin{document}

   
      \title[Fractional integral operators on homogeneous groups]
   {Hardy-Littlewood, Bessel-Riesz, and fractional integral operators in anisotropic Morrey and Campanato spaces}

\author[M. Ruzhansky]{Michael Ruzhansky}
\address{
  Michael Ruzhansky:
  \endgraf
  Department of Mathematics
  \endgraf
  Imperial College London
  \endgraf
  180 Queen's Gate, London SW7 2AZ
  \endgraf
  United Kingdom
  \endgraf
  {\it E-mail address} {\rm m.ruzhansky@imperial.ac.uk}
  }
\author[D. Suragan]{Durvudkhan Suragan}
\address{
  Durvudkhan Suragan:
  \endgraf
  Institute of Mathematics and Mathematical Modelling
  \endgraf
  125 Pushkin str.
  \endgraf
  050010 Almaty
  \endgraf
  Kazakhstan
  \endgraf
  {\it E-mail address} {\rm suragan@math.kz}
  }
\author[N. Yessirkegenov]{Nurgissa Yessirkegenov}
\address{
  Nurgissa Yessirkegenov:
  \endgraf
  Institute of Mathematics and Mathematical Modelling
  \endgraf
  125 Pushkin str.
  \endgraf
  050010 Almaty
  \endgraf
  Kazakhstan
  \endgraf
  and
  \endgraf
  Department of Mathematics
  \endgraf
  Imperial College London
  \endgraf
  180 Queen's Gate, London SW7 2AZ
  \endgraf
  United Kingdom
  \endgraf
  {\it E-mail address} {\rm n.yessirkegenov15@imperial.ac.uk}
  }

\thanks{The authors were supported in parts by the EPSRC
 grant EP/K039407/1 and by the Leverhulme Grant RPG-2014-02,
 as well as by the MESRK grant 5127/GF4. No new data was collected or generated during the course of research.}


     \begin{abstract} We analyse Morrey spaces, generalised Morrey spaces and Campanato spaces on homogeneous groups. The boundedness of the Hardy-Littlewood maximal operator, Bessel-Riesz operators, generalised Bessel-Riesz operators and generalised fractional integral operators in generalised Morrey spaces on homogeneous groups is shown. Moreover, we prove the boundedness of the modified version of the generalised fractional integral operator and Olsen type inequalities in Campanato spaces and generalised Morrey spaces on homogeneous groups, respectively. Our results extend results known in the isotropic Euclidean settings, however, some of them are new already in the standard Euclidean cases.
     \end{abstract}
     \maketitle

\section{Introduction}
Consider the following Bessel-Riesz operators
\begin{equation}\label{intro1}
I_{\alpha,\gamma}f(x)=\int_{\Rn}K_{\alpha,\gamma}(x-y)f(y)dy=\int_{\Rn}\frac{|x-y|^{\alpha-n}}{(1+|x-y|)^{\gamma}}f(y)dy,
\end{equation}
where $f\in L^{p}_{loc}(\Rn), p\geq 1, \gamma \geq0$ and $0<\alpha<n$. Here, $I_{\alpha,\gamma}$ and $K_{\alpha,\gamma}$ are called Bessel-Riesz operator and Bessel-Riesz kernel, respectively.
The boundedness of the Bessel-Riesz operators on Lebesgue spaces was shown by Hardy and Littlewood in \cite{HL27}, \cite{HL32} and Sobolev in \cite{Sob38}. In the case of $\Rn$, the Hardy-Littlewood maximal operator, the Riesz potential $I_{\alpha,0}=I_{\alpha}$, the generalised fractional integral operators, which are a generalised form of the Riesz potential $I_{\alpha,0}=I_{\alpha}$, Bessel-Riesz operators and Olsen type inequalities are widely analysed on Lebesgue spaces, Morrey spaces and generalised Morrey spaces (see e.g. \cite{Adams75}, \cite{CF87}, \cite{Nakai94}, \cite{EGN04}, \cite{Eridani02}, \cite{KNS99}, \cite{Nakai01}, \cite{Nakai02}, \cite{GE09}, \cite{SST12}, \cite{IGLE15} and \cite{IGE16}, as well as \cite{Bur13} for a recent survey).
For some of their functional analytic properties see also \cite{BDN13, BNC14} and references therein.

In this paper we are interested in the boundedness of the Hardy-Littlewood maximal operator, Bessel-Riesz operators, generalised Bessel-Riesz operators, generalised fractional integral operators and Olsen type inequalities in generalised Morrey spaces on homogeneous Lie groups. The obtained results give new statements already in the Euclidean setting of $\Rn$ when we are working with anisotropic differential structure. Furthermore, even in the isotropic situation in $\Rn$, one novelty of all the obtained results is also in the arbitrariness of the choice of any homogeneous quasi-norm, and some estimates are also new in the usual isotropic structure of $\Rn$ with the Euclidean norm, which we will be indicating at relevant places. 

Thus, we could have worked directly in $\Rn$ with anisotropic structure, but since the methods work equally well in the setting of Folland and Stein's homogeneous groups, we formulate all the results in such (greater) generality. In particular, it follows the general strategy initiated by their work, of distilling results of harmonic analysis depending only on the group and dilation structures: in this respect the present paper shows that the harmonic analysis on Morrey spaces largely falls into this category.

We refer to recent papers \cite{Ruzhansky-Suragan:squares}, \cite{Ruzhansky-Suragan:horizontal}, \cite{Ruzhansky-Suragan:L2-CKN}, and \cite{Ruzhansky-Suragan:layers} for discussions related to different functional inequalities with special as well as arbitrary homogeneous quasi-norms in different settings.  Morrey spaces for non-Euclidean distances find their applications in many problems, see e.g. \cite{GS15a, GS15b} and \cite{GS16}.

For the convenience of the reader let us now shortly recapture the main results of this paper. 

For the definitions of the spaces appearing in the formulations below, see 
\eqref{Morreydef} for Morrey spaces $L^{p,q}(\mathbb G)$,
\eqref{GMorrey} for generalised Morrey spaces $L^{p,\phi}(\mathbb G)$, and 
\eqref{GCampanato} for generalised Camponato spaces $\mathcal{L}^{p,\phi}(\mathbb{G})$,
as well as \eqref{HLmax} for the Hardy-Littlewood maximal operator $M$,
\eqref{BRdef} for Bessel-Riesz operators $I_{\alpha,\gamma}$,
\eqref{I_rho} for generalised Bessel-Riesz operators $I_{\rho,\gamma}$, and 
\eqref{T_rho} for generalised fractional intergral operators $T_\rho$.
Both $I_{\rho,\gamma}$ and $T_\rho$ generalise the Riesz transform and the Bessel-Riesz transform in different directions.

Thus, in this paper we show that for a homogeneous group $\mathbb{G}$ of homogeneous dimension $Q$ and any homogeneous quasi-norm $|\cdot|$ we have the following properties:
\begin{itemize}
\item If $0<\alpha<Q$ and $\gamma>0$, then $K_{\alpha, \gamma}\in L^{p_{1}}(\mathbb{G})$ for $\frac{Q}{Q+\gamma-\alpha}<p_{1}<\frac{Q}{Q-\alpha}$, and
$$\|K_{\alpha, \gamma}\|_{L^{p_{1}}(\mathbb{G})}\thicksim\left(\sum_{k\in\mathbb{Z}}\frac{(2^{k}R)^{(\alpha-Q)p_{1}+Q}}{(1+2^{k}R)^{\gamma p_{1}}}\right)^{\frac{1}{p_{1}}}$$
for any $R>0$, where $K_{\alpha, \gamma}:=\frac{|x|^{\alpha-n}}{(1+|x|)^{\gamma}}$.
\item For any $f\in L^{p,\phi}(\mathbb{G})$ and $1<p<\infty$, we have
$$
\|Mf\|_{L^{p,\phi}(\mathbb{G})}\leq C_{p}\|f\|_{L^{p,\phi}(\mathbb{G})},
$$
where generalised Morrey space $L^{p,\phi}(\mathbb{G})$ and Hardy-Littlewood maximal operator $Mf$ are defined in \eqref{GMorrey} and \eqref{HLmax}, respestively.
\item Let $\gamma>0$ and $0<\alpha<Q$. If $\phi(r)\leq Cr^{\beta}$ for every $r>0, \beta<-\alpha, 1<p<\infty$, and $\frac{Q}{Q+\gamma-\alpha}<p_{1}<\frac{Q}{Q-\alpha}$, then for all $f\in L^{p, \phi}(\mathbb{G})$ we have
$$\|I_{\alpha,\gamma}f\|_{L^{q, \psi}(\mathbb{G})}\leq C_{p, \phi, Q}\|K_{\alpha,\gamma}\|_{L^{p_{1}}(\mathbb{G})}\|f\|_{L^{p, \phi}(\mathbb{G})},
$$
where $q=\frac{\beta p_{1}^{'} p}{\beta p_{1}^{'}+Q}$ and $\psi(r)=\phi(r)^{p/q}$. The Bessel-Riesz operator $I_{\alpha,\gamma}$ on a homogenous group is defined in \eqref{BRdef}.
\item Let $\gamma>0$ and $0<\alpha<Q$. If $\phi(r)\leq Cr^{\beta}$ for every $r>0, \beta<-\alpha, $ $\frac{Q}{Q+\gamma-\alpha}<p_{2}\leq p_{1}<\frac{Q}{Q-\alpha}$ and $p_{2}\geq1$, then for all $f\in L^{p, \phi}(\mathbb{G})$ we have
$$\|I_{\alpha,\gamma}f\|_{L^{q, \psi}(\mathbb{G})}\leq C_{p, \phi, Q}\|K_{\alpha,\gamma}\|_{L^{p_{2},p_{1}}(\mathbb{G})}\|f\|_{L^{p, \phi}(\mathbb{G})},
$$
where $1<p<\infty, q=\frac{\beta p_{1}^{'} p}{\beta p_{1}^{'}+Q}, \psi(r)=\phi(r)^{p/q}$.
\item Let $\omega:\mathbb{R}^{+}\rightarrow\mathbb{R}^{+}$ satisfy the doubling condition and assume that $\omega(r)\leq Cr^{-\alpha}$ for every $r>0$, so that $K_{\alpha,\gamma}\in L^{p_{2}, \omega}(\mathbb{G})$ for $\frac{Q}{Q+\gamma-\alpha}<p_{2}<\frac{Q}{Q-\alpha}$ and $p_{2}\geq1$, where $0<\alpha<Q$ and $\gamma>0$. If $\phi(r)\leq Cr^{\beta}$ for every $r>0$, where $\beta<-\alpha<-Q-\beta$, then for all $f\in L^{p, \phi}(\mathbb{G})$ we have
$$\|I_{\alpha,\gamma}f\|_{L^{q, \psi}(\mathbb{G})}\leq C_{p, \phi, Q}\|K_{\alpha,\gamma}\|_{L^{p_{2},\omega}(\mathbb{G})}\|f\|_{L^{p, \phi}(\mathbb{G})},
$$
where $1< p<\infty, q=\frac{\beta p}{\beta+Q-\alpha}$ and $ \psi(r)=\phi(r)^{p/q}$.
\item Let $\gamma>0$ and let $\rho$ and $\phi$ satisfy the doubling condition \eqref{Pre02}. 
Let $1<p<q<\infty$.
Let $\phi$ be surjective and satisfy
$$\int_{r}^{\infty}\frac{\phi(t)^{p}}{t}dt\leq C_{1}\phi(r)^{p}, 
$$
and
$$\phi(r)\int_{0}^{r}\frac{\rho(t)}{t^{\gamma-Q+1}}dt+\int_{r}^{\infty}
\frac{\rho(t)\phi(t)}{t^{\gamma-Q+1}}dt\leq C_{2} \phi(r)^{p/q},
$$
for all $r>0$. Then we have
$$
\|I_{\rho, \gamma}f\|_{L^{q,\phi^{p/q}}(\mathbb{G})}\leq C_{p, q, \phi, Q}\|f\|_{L^{p,\phi}(\mathbb{G})},
$$
where the generalised Bessel-Riesz operator $I_{\rho, \gamma}$ is defined in \eqref{I_rho}.
This result is new already in the standard setting of $\Rn$.
\item Let $\rho$ and $\phi$ satisfy the doubling condition \eqref{Pre02}. Let $\gamma>0$, and assume that $\phi$ is surjective and satisfies \eqref{I_Gfracopthm1}-\eqref{I_Gfracopthm2}. Then we have
$$\|W\cdot I_{\rho, \gamma}f\|_{L^{p,\phi}(\mathbb{G})}\leq C_{p, \phi, Q}\|W\|_{L^{p_{2},\phi^{p/p_{2}}}(\mathbb{G})}\|f\|_{L^{p,\phi}(\mathbb{G})}, \quad 1<p<p_{2}<\infty,
$$
provided that $W\in L^{p_{2},\phi^{p/p_{2}}}(\mathbb{G})$.
This result is new even in the Euclidean cases.
\item Let $\rho$ and $\phi$ satisfy the doubling condition \eqref{Pre02}. 
Let $1<p<q<\infty$.
Let $\phi$ be surjective and satisfy
$$\int_{r}^{\infty}\frac{\phi(t)^{p}}{t}dt\leq C_{1}\phi(r)^{p}, 
$$
and
$$\phi(r)\int_{0}^{r}\frac{\rho(t)}{t}dt+\int_{r}^{\infty}\frac{\rho(t)\phi(t)}{t}dt\leq C_{2} \phi(r)^{p/q},
$$
for all $r>0$. Then we have
$$
\|T_{\rho}f\|_{L^{q,\phi^{p/q}}(\mathbb{G})}\leq C_{p, q, \phi, Q}\|f\|_{L^{p,\phi}(\mathbb{G})}, 
$$
where the generalised fractional integral operator $T_{\rho}$ is defined in \eqref{T_rho}.
\item Let $\rho$ and $\phi$ satisfy the doubling condition \eqref{Pre02}. Let $\phi$ be surjective and satisfy \eqref{Gfracopthm1}-\eqref{Gfracopthm2}. Then we have
$$\|W\cdot T_{\rho}f\|_{L^{p,\phi}(\mathbb{G})}\leq C_{p, \phi, Q}\|W\|_{L^{p_{2},\phi^{p/p_{2}}}(\mathbb{G})}\|f\|_{L^{p,\phi}(\mathbb{G})}, \quad 1<p<p_{2}<\infty,
$$
provided that $W\in L^{p_{2},\phi^{p/p_{2}}}(\mathbb{G})$.
\item Let $\omega:\mathbb{R}^{+}\rightarrow\mathbb{R}^{+}$ satisfy the doubling condition and assume that $\omega(r)\leq Cr^{-\alpha}$ for every $r>0$, so that $K_{\alpha,\gamma}\in L^{p_{2}, \omega}(\mathbb{G})$ for $\frac{Q}{Q+\gamma-\alpha}<p_{2}<\frac{Q}{Q-\alpha}$ and $p_{2}\geq1$, where $0<\alpha<Q$, $1<p<\infty, q=\frac{\beta p}{\beta+Q-\alpha}$ and $\gamma>0$. If $\phi(r)\leq Cr^{\beta}$ for every $r>0$, where $\beta<-\alpha<-Q-\beta$, then we have
$$\|W\cdot I_{\alpha,\gamma}f\|_{L^{p,\phi}(\mathbb{G})}\leq C_{p, \phi, Q}\|W\|_{L^{p_{2},\phi^{p/p_{2}}}(\mathbb{G})}\|f\|_{L^{p,\phi}(\mathbb{G})},
$$
provided that $W\in L^{p_{2},\phi^{p/p_{2}}}(\mathbb{G})$, where $\frac{1}{p_{2}}=\frac{1}{p}-\frac{1}{q}$. This result is new already in the Euclidean setting of $\mathbb R^n$.
\item Let $\rho$ satisfy \eqref{Pre01}, \eqref{Pre02}, \eqref{Pre03}, \eqref{Pre04}, and let $\phi$ satisfy the doubling condition \eqref{Pre02} and $\int_{1}^{\infty}\frac{\phi(t)}{t}dt<\infty$. If
$$\int_{r}^{\infty}\frac{\phi(t)}{t}dt\int_{0}^{r}\frac{\rho(t)}{t}dt+r\int_{r}^{\infty}\frac{\rho(t)\phi(t)}{t^{2}}dt\leq C_{3} \psi(r) \;\;{\rm for \;\;all} \;\;r>0,
$$
then we have
$$
\|\widetilde{T}_{\rho}f\|_{\mathcal{L}^{p,\psi}(\mathbb{G})}\leq C_{p, \phi, Q}\|f\|_{\mathcal{L}^{p,\phi}(\mathbb{G})}, \;\;1<p<\infty,
$$
where the generalised Campanato space $\mathcal{L}^{p,\psi}(\mathbb{G})$ and operator $\widetilde{T}_{\rho}$ are defined in \eqref{GCampanato} and \eqref{ModT}, respectively.
\end{itemize}
This paper is structured as follows. In Section \ref{SEC:prelim} we briefly recall the concepts of homogeneous groups and fix the notation. The boundedness of the Hardy-Littlewood maximal operator and Bessel-Riesz operators in generalised Morrey spaces on homogeneous groups is proved in Section \ref{SEC:Hardy-Littlewood} and in Section \ref{SEC:Bessel-Riesz}, respectively. In Section \ref{SEC:I_Gfracop} we prove the boundedness of the generalised Bessel-Riesz operators and Olsen type inequality for these operators in generalised Morrey spaces on homogeneous groups. The boundedness of the generalised fractional integral operators and Olsen type inequality for these operators in generalised Morrey spaces on homogeneous groups are proved in Section \ref{SEC:Gfracop}. Finally, in Section \ref{SEC:Campanato} we investigate the boundedness of the modified version of the generalised fractional integral operator in Campanato spaces on homogeneous groups.
\section{Preliminaries}
\label{SEC:prelim}
A connected simply connected Lie group $\mathbb G$ is called a {\em homogeneous group} if
its Lie algebra $\mathfrak{g}$ is equipped with a family of dilations:
$$D_{\lambda}={\rm Exp}(A \,{\rm ln}\lambda)=\sum_{k=0}^{\infty}
\frac{1}{k!}({\rm ln}(\lambda) A)^{k},$$
where $A$ is a diagonalisable positive linear operator on $\mathfrak{g}$,
and each $D_{\lambda}$ is a morphism of $\mathfrak{g}$,
that is, 
$$\forall X,Y\in \mathfrak{g},\, \lambda>0,\;
[D_{\lambda}X, D_{\lambda}Y]=D_{\lambda}[X,Y].$$

The exponential mapping $\exp_{\mathbb G}:\mathfrak g\to\mathbb G$ is a global diffeomorphism and gives the dilation structure, which is denoted by $D_{\lambda}x$ or just by $\lambda x$, on $\mathbb G$.

Then we have
\begin{equation}
|D_{\lambda}(S)|=\lambda^{Q}|S| \quad {\rm and}\quad \int_{\mathbb{G}}f(\lambda x)
dx=\lambda^{-Q}\int_{\mathbb{G}}f(x)dx,
\end{equation}
where $dx$ is the Haar measure on $\mathbb{G}$, $|S|$ is the volume of a measurable set $S\subset \mathbb{G}$ and $Q := {\rm Tr}\,A$ is the homogeneous dimension of $\mathbb G$. Recall that the Haar measure on a homogeneous group $\mathbb{G}$ is the standard Lebesgue measure for $\Rn$ (see e.g. \cite[Proposition 1.6.6]{FR}).

Let $|\cdot|$ be a homogeneous quasi-norm on $\mathbb G$.
We will denote the quasi-ball centred at $x\in\mathbb{G}$ with radius $R > 0$ by
$$B(x,R):=\{y\in \mathbb{G}: |x^{-1}y|<R\}$$
and we will also use the notation 
$$B^{c}(x,R):=\{y\in \mathbb{G}: |x^{-1}y|\geq R\}.$$
The proof of the following important polar decomposition on homogeneous Lie groups was given
by Folland and Stein \cite{FS-Hardy}, which can be also found in
\cite[Section 3.1.7]{FR}:
there is a (unique)
positive Borel measure $\sigma$ on the
unit sphere
\begin{equation}\label{EQ:sphere}
\wp:=\{x\in \mathbb{G}:\,|x|=1\},
\end{equation}
so that for any $f\in L^{1}(\mathbb{G})$, one has
\begin{equation}\label{EQ:polar}
\int_{\mathbb{G}}f(x)dx=\int_{0}^{\infty}
\int_{\wp}f(ry)r^{Q-1}d\sigma(y)dr.
\end{equation}
Now, for any $f\in L^{p}_{loc}(\mathbb{G}), \;p\geq1$ and $\gamma\geq0, \;0<\alpha<Q$, we shall define the Bessel-Riesz operators on homogeneous groups by
\begin{equation}\label{BRdef}
I_{\alpha,\gamma}f(x):=\int_{\mathbb{G}}K_{\alpha,\gamma}(xy^{-1})f(y)dy=
\int_{\mathbb{G}}\frac{|xy^{-1}|^{\alpha-Q}}{(1+|xy^{-1}|)^{\gamma}}f(y)dy,
\end{equation}
where $|\cdot|$ is any homogeneous quasi-norm. Here, $K_{\alpha,\gamma}$ is the Bessel-Riesz kernel. Hereafter, $C$, $C_{i}$, $C_{p}$, $C_{p,\phi, Q}$ and $C_{p, q, \phi, Q}$ are positive constants, which are not necessarily the same from line to line.

Let us recall the following result, which will be used in the sequel.

\begin{lem}[\cite{IGLE15}] \label{Pre0} If $b>a>0$ then $\underset{k\in\mathbb{Z}}{\sum}\frac{(u^{k}R)^{a}}{(1+u^{k}R)^{b}}<\infty,$ for every $u>1$ and $R>0$.
\end{lem}
We now calculate the $L^{p}$-norms of the Bessel-Riesz kernel.

\begin{thm}\label{Pre1} Let $\mathbb{G}$ be a homogeneous group
of homogeneous dimension $Q$. Let $|\cdot|$ be a homogeneous quasi-norm. Let $K_{\alpha, \gamma}(x)=\frac{|x|^{\alpha-Q}}{(1+|x|)^{\gamma}}$. If $0<\alpha<Q$ and $\gamma>0$ then $K_{\alpha, \gamma}\in L^{p_{1}}(\mathbb{G})$ and
$$\|K_{\alpha, \gamma}\|_{L^{p_{1}}(\mathbb{G})}\thicksim\left(\sum_{k\in\mathbb{Z}}\frac{(2^{k}R)^{(\alpha-Q)p_{1}+Q}}{(1+2^{k}R)^{\gamma p_{1}}}\right)^{\frac{1}{p_{1}}},$$
for $\frac{Q}{Q+\gamma-\alpha}<p_{1}<\frac{Q}{Q-\alpha}$.
\end{thm}
\begin{proof}[Proof of Theorem \ref{Pre1}]
Introducing polar coordinates $(r,y)=(|x|, \frac{x}{\mid x\mid})\in (0,\infty)\times\wp$ on $\mathbb{G}$, where $\wp$ is the sphere as in \eqref{EQ:sphere}, and using \eqref{EQ:polar} for any $R>0$, we have
$$\int_{\mathbb{G}}|K_{\alpha, \gamma}(x)|^{p_{1}}dx=\int_{\mathbb{G}}\frac{|x|^{(\alpha-Q)p_{1}}}{(1+|x|)^{\gamma p_{1}}}dx$$
$$=\int_{0}^{\infty}\int_{\wp}\frac{r^{(\alpha-Q)p_{1}+Q-1}}{(1+r)^{\gamma p_{1}}}d\sigma(y)dr=
|\sigma|\sum_{k\in\mathbb{Z}}\int_{2^{k}R\leq r<2^{k+1}R}\frac{r^{(\alpha-Q)p_{1}+Q-1}}{(1+r)^{\gamma p_{1}}}dr,$$
where $|\sigma|$ is the $Q-1$ dimensional surface measure of the unit sphere.

Then it follows that
$$\int_{\mathbb{G}}|K_{\alpha, \gamma}(x)|^{p_{1}}dx\leq|\sigma|\sum_{k\in\mathbb{Z}}\frac{1}{(1+2^{k}R)^{\gamma p_{1}}}
\int_{2^{k}R\leq r<2^{k+1}R}r^{(\alpha-Q)p_{1}+Q-1}dr$$$$=\frac{|\sigma|(2^{(\alpha-Q)p_{1}+Q}-1)}{(\alpha-Q)p_{1}+Q}
\sum_{k\in\mathbb{Z}}\frac{(2^{k}R)^{(\alpha-Q)p_{1}+Q}}{(1+2^{k}R)^{\gamma p_{1}}}.$$
On the other hand, we obtain
$$\int_{\mathbb{G}}|K_{\alpha, \gamma}(x)|^{p_{1}}dx\geq\frac{|\sigma|}{2^{\gamma p_{1}}}\sum_{k\in\mathbb{Z}}
\frac{1}{(1+2^{k}R)^{\gamma p_{1}}}\int_{2^{k}R\leq r<2^{k+1}R}r^{(\alpha-Q)p_{1}+Q-1}dr$$
$$=\frac{|\sigma|(2^{(\alpha-Q)p_{1}+Q}-1)}{2^{\gamma p_{1}}((\alpha-Q)p_{1}+Q)}\sum_{k\in\mathbb{Z}}\frac{(2^{k}R)^{(\alpha-Q)p_{1}+Q}}{(1+2^{k}R)^{\gamma p_{1}}}.$$
Therefore, for every $R>0$ we arrive at
$$\int_{\mathbb{G}}|K_{\alpha, \gamma}(x)|^{p_{1}}dx\thicksim\sum_{k\in\mathbb{Z}}\frac{(2^{k}R)^{(\alpha-Q)p_{1}+Q}}{(1+2^{k}R)^{\gamma p_{1}}}.$$
For $p_{1}\in\left(\frac{Q}{Q+\gamma-\alpha},\frac{Q}{Q-\alpha}\right)$ using Lemma \ref{Pre0} with $u=2, a=(\alpha-Q)p_{1}+Q, b=\gamma p_{1}$, we obtain $\sum_{k\in\mathbb{Z}}\frac{(2^{k}R)^{(\alpha-Q)p_{1}+Q}}{(1+2^{k}R)^{\gamma p_{1}}}<\infty$ which implies $K_{\alpha, \gamma}\in L^{p_{1}}(\mathbb{G})$.
\end{proof}
The following is well-known on homogeneous groups, see e.g. \cite[Proposition 1.5.2]{FR}.

\begin{prop}[Young's inequality] \label{Young}  Let $\mathbb{G}$ be a homogeneous group. Suppose
$1\leq p,q,p_{1}\leq\infty$ and $\frac{1}{q}+1=\frac{1}{p}+\frac{1}{p_{1}}$.
If $f\in L^{p}(\mathbb{G})$ and $g\in L^{p_{1}}(\mathbb{G})$ then
$$\|g\ast f\|_{L^{q}(\mathbb{G})}\leq \|f\|_{L^{p}(\mathbb{G})}\|g\|_{L^{p_{1}}(\mathbb{G})}.$$
\end{prop}
In view of Proposition \ref{Young} and taking into account the definition of Bessel-Riesz operator \ref{BRdef}, we immediately get the following Corollary \ref{Pre2}:
\begin{cor} \label{Pre2} Let $\mathbb{G}$ be a homogeneous group
of homogeneous dimension $Q$. Let $|\cdot|$ be a homogeneous quasi-norm. Then for $0<\alpha<Q, \gamma>0$, we have
$$\|I_{\alpha,\gamma}f\|_{L^{q}(\mathbb{G})}\leq \|K_{\alpha,\gamma}\|_{L^{p_{1}}(\mathbb{G})}\|f\|_{L^{p}(\mathbb{G})}$$
for every $f\in L^{p}(\mathbb{G})$ where $1\leq p,q,p_{1}\leq\infty$, $\frac{1}{q}+1=\frac{1}{p}+\frac{1}{p_{1}}$ and $\frac{Q}{Q+\gamma-\alpha}<p_{1}<\frac{Q}{Q-\alpha}.$
\end{cor}
Corollary \ref{Pre2} shows that the $I_{\alpha,\gamma}$ is bounded from $L^{p}(\mathbb{G})$ to $L^{q}(\mathbb{G})$ and
$$\|I_{\alpha,\gamma}\|_{L^{p}(\mathbb{G}) \to L^{q}(\mathbb{G})} \leq\|K_{\alpha, \gamma}\|_{L^{p_{1}}(\mathbb{G})}.$$

\section{The boundedness of Hardy-Littlewood maximal operator in generalised Morrey spaces}
\label{SEC:Hardy-Littlewood}

In this section we define Morrey and generalised Morrey spaces on homogeneous groups. Then we prove that the Hardy-Littlewood maximal operator is bounded in these spaces. Note that in the isotropic Abelian case the result was obtained by Nakai \cite{Nakai94}. Let $\mathbb G$ be a homogeneous group of homogeneous dimension $Q$.

Let us define the Morrey spaces $L^{p,q}(\mathbb{G})$ by
\begin{equation}\label{Morreydef}L^{p,q}(\mathbb{G}):=\{f\in L^{p}_{loc}(\mathbb{G}):\|f\|_{L^{p,q}(\mathbb{G})}<\infty\}, \;\;1\leq p \leq q, \end{equation}
where
$$\|f\|_{L^{p,q}(\mathbb{G})}:=\sup_{r>0}r^{Q(1/q-1/p)}\left(\int_{B(0,r)}|f(x)|^{p}dx\right)^{1/p}.$$
Next, for a function $\phi:\mathbb{R}^{+}\rightarrow\mathbb{R}^{+}$ and $1\leq p<\infty$, we define the generalised Morrey space $L^{p,\phi}(\mathbb{G})$ by 
\begin{equation}\label{GMorrey}L^{p,\phi}(\mathbb{G}):=\{f\in L^{p}_{loc}(\mathbb{G}):\|f\|_{L^{p,\phi}(\mathbb{G})}<\infty\},
\end{equation}
where
$$\|f\|_{L^{p,\phi}(\mathbb{G})}:=\sup_{r>0}\frac{1}{\phi(r)}\left(\frac{1}{r^{Q}}\int_{B(0,r)}|f(x)|^{p}dx\right)^{1/p}.$$
Here we assume that $\phi$ is nonincreasing and $t^{Q/p}\phi(t)$ is nondecreasing, so that $\phi$ satisfies the doubling condition, i.e. there exists a constant $C_{1}>0$ such that
\begin{equation}\label{Pre02}
\frac{1}{2}\leq \frac{r}{s}\leq 2 \Longrightarrow \frac{1}{C_{1}}\leq \frac{\rho(r)}{\rho(s)}\leq C_{1}.
\end{equation}
Now, for every $f\in L^{p}_{loc}(\mathbb{G})$, we define the Hardy-Littlewood maximal operator $M$ by
\begin{equation}\label{HLmax}Mf(x):=\sup_{x\in{B}}\frac{1}{|B(0,r)|}\int_{B(0,r)}|f(y)|dy, \;x\in\mathbb{G},\end{equation}
where $|B(0,r)|$ denotes the Haar measure of the ball $B=B(0,r)$.

Using the definition of Morrey spaces \eqref{Morreydef}, one can readily obtain the following Lemma \ref{incul_Mor_lem}:
\begin{lem}\label{incul_Mor_lem} Let $\mathbb{G}$ be a homogeneous group
of homogeneous dimension $Q$. Let $|\cdot|$ be a homogeneous quasi-norm. Then
\begin{equation}\label{inclus_Mor}\|K_{\alpha,\gamma}\|_{L^{p_{2},p_{1}}(\mathbb{G})}\leq\|K_{\alpha,\gamma}\|_{L^{p_{1},p_{1}}(\mathbb{G})}
=\|K_{\alpha,\gamma}\|_{L^{p_{1}}(\mathbb{G})},
\end{equation}
where $1\leq p_{2}\leq p_{1}$ and $\frac{Q}{Q+\gamma-\alpha}<p_{1}<\frac{Q}{Q-\alpha}$.
\end{lem}
We now prove the boundedness of the Hardy-Littlewood maximal operator on generalised Morrey spaces.

\begin{thm}\label{Nakaithm} Let $\mathbb{G}$ be a homogeneous group.
For any $f\in L^{p,\phi}(\mathbb{G})$ and $1<p<\infty$, we have
\begin{equation}\label{Nakai}
\|Mf\|_{L^{p,\phi}(\mathbb{G})}\leq C_{p}\|f\|_{L^{p,\phi}(\mathbb{G})}.
\end{equation}
\end{thm}
\begin{proof}[Proof of Theorem \ref{Nakai}] By the definition of the norm of the generalised Morrey space \eqref{GMorrey}, we have
$$\|f\|_{L^{p,\phi}(\mathbb{G})}=\sup_{r>0}\frac{1}{\phi(r)}\left(\frac{1}{r^{Q}}\int_{B(0,r)}|f(x)|^{p}dx\right)^{1/p}.$$
This implies that
\begin{equation}\label{HLmax1}
\left(\int_{B(0,r)}|f(x)|^{p}dx\right)^{1/p}\leq\phi(r)r^{\frac{Q}{p}}\|f\|_{L^{p,\phi}(\mathbb{G})},
\end{equation}
for any $r>0$.

On the other hand, using Corollary 2.5 (b) from Folland and Stein \cite{FS-Hardy} we have
\begin{equation}\label{HLmax2}
\left(\int_{B(0,r)}|Mf(x)|^{p}dx\right)^{1/p}\leq C_{p}\left(\int_{B(0,r)}|f(x)|^{p}dx\right)^{1/p}.
\end{equation}
Combining \eqref{HLmax1} and \eqref{HLmax2} we arrive at
$$\frac{1}{\phi(r)}\left(\frac{1}{r^{Q}}\int_{B(0,r)}|Mf(x)|^{p}dx\right)^{1/p}\leq C_{p}\|f\|_{L^{p,\phi}(\mathbb{G})},$$
for all $r>0$. Thus
$$\|Mf\|_{L^{p,\phi}(\mathbb{G})}\leq C_{p}\|f\|_{L^{p,\phi}(\mathbb{G})},$$
completing the proof.
\end{proof}

\section{Inequalities for Bessel-Riesz operators on generalised Morrey spaces}
\label{SEC:Bessel-Riesz}
In this section, we prove the boundedness of the Bessel-Riesz operators on generalised Morrey spaces \eqref{GMorrey}.
\begin{thm}\label{BRth1}
Let $\mathbb{G}$ be a homogeneous group
of homogeneous dimension $Q$. Let $|\cdot|$ be a homogeneous quasi-norm. Let $\gamma>0$ and $0<\alpha<Q$. If $\phi(r)\leq Cr^{\beta}$ for every $r>0, \beta<-\alpha, 1<p<\infty$, and $\frac{Q}{Q+\gamma-\alpha}<p_{1}<\frac{Q}{Q-\alpha}$, then for all $f\in L^{p, \phi}(\mathbb{G})$ we have
\begin{equation}\label{GMth1}\|I_{\alpha,\gamma}f\|_{L^{q, \psi}(\mathbb{G})}\leq C_{p, \phi, Q}\|K_{\alpha,\gamma}\|_{L^{p_{1}}(\mathbb{G})}\|f\|_{L^{p, \phi}(\mathbb{G})},
\end{equation}
where $q=\frac{\beta p_{1}^{'} p}{\beta p_{1}^{'}+Q}$ and $\psi(r)=\phi(r)^{p/q}$.
\end{thm}
\begin{proof}[Proof of Theorem \ref{BRth1}] For every $f\in L^{p, \phi}(\mathbb{G})$, let us write $I_{\alpha,\gamma}f(x)$ in the form
$$I_{\alpha,\gamma}f(x):=I_{1}(x)+I_{2}(x),$$
where $I_{1}(x):=\int_{B(x,R)}\frac{|xy^{-1}|^{\alpha-Q}f(y)}{(1+|xy^{-1}|)^{\gamma}}dy$ and $I_{2}(x):=\int_{B^{c}(x,R)}\frac{|xy^{-1}|^{\alpha-Q}f(y)}{(1+|xy^{-1}|)^{\gamma}}dy$, for some $R>0$.

By using dyadic decomposition for $I_{1}$, we obtain
\begin{align*}|I_{1}(x)|&\leq\sum_{k=-\infty}^{-1}\int_{2^{k}R\leq |xy^{-1}|<2^{k+1}R}\frac{|xy^{-1}|^{\alpha-Q}|f(y)|}{(1+|xy^{-1}|)^{\gamma}}dy\\&
\leq \sum_{k=-\infty}^{-1}\frac{(2^{k}R)^{\alpha-Q}}{(1+2^{k}R)^{\gamma}}\int_{2^{k}R\leq |xy^{-1}|<2^{k+1}R}
|f(y)|dy\\&
\leq CMf(x)\sum_{k=-\infty}^{-1}\frac{(2^{k}R)^{\alpha-Q+Q/p_{1}}(2^{k}R)^{Q/p_{1}^{'}}}{(1+2^{k}R)^{\gamma}}.
\end{align*}
From this using H\"{o}lder inequality for $\frac{1}{p_{1}}+\frac{1}{p_{1}^{'}}=1$ we get
$$|I_{1}(x)|\leq CMf(x)\left(\sum_{k=-\infty}^{-1}\frac{(2^{k}R)^{(\alpha-Q)p_{1}+Q}}{(1+2^{k}R)^{\gamma p_{1}}}\right)^{1/p_{1}}\left(\sum_{k=-\infty}^{-1}(2^{k}R)^{Q}\right)^{1/p_{1}^{'}}.$$
Since
\begin{equation}\label{sum_equiv1}
\left(\sum_{k=-\infty}^{-1}\frac{(2^{k}R)^{(\alpha-Q)p_{1}+Q}}{(1+2^{k}R)^{\gamma p_{1}}}\right)^{1/p_{1}}
\leq \left(\sum_{k\in\mathbb{Z}}\frac{(2^{k}R)^{(\alpha-Q)p_{1}+Q}}{(1+2^{k}R)^{\gamma p_{1}}}\right)^{1/p_{1}}\sim\|K_{\alpha,\gamma}\|_{L^{p_{1}}(\mathbb{G})},
\end{equation}
we arrive at
\begin{equation}\label{GM1}|I_{1}(x)|\leq C\|K_{\alpha,\gamma}\|_{L^{p_{1}}(\mathbb{G})}
Mf(x)R^{Q/p_{1}^{'}}.\end{equation}
For the second term $I_{2}$, by using H\"{o}lder inequality for $\frac{1}{p}+\frac{1}{p^{'}}=1$ we obtain that
\begin{align*}
|I_{2}(x)|& \leq \sum_{k=0}^{\infty}\frac{(2^{k}R)^{\alpha-Q}}{(1+2^{k}R)^{\gamma}}\int_{2^{k}R\leq|xy^{-1}|<2^{k+1}R}|f(y)|dy
 \\&\leq \sum_{k=0}^{\infty}\frac{(2^{k}R)^{\alpha-Q}}{(1+2^{k}R)^{\gamma}}
 \left(\int_{2^{k}R\leq|xy^{-1}|<2^{k+1}R}dy\right)^{1/p^{'}}
\left(\int_{2^{k}R\leq|xy^{-1}|<2^{k+1}R}|f(y)|^{p}dy\right)^{1/p}
 \\ & =\sum_{k=0}^{\infty}\frac{(2^{k}R)^{\alpha-Q}}{(1+2^{k}R)^{\gamma}}
 \left(\int_{2^{k}R}^{2^{k+1}R}\int_{\wp}r^{Q-1}d\sigma(y)dr\right)^{1/p^{'}}
\left(\int_{2^{k}R\leq|xy^{-1}|<2^{k+1}R}|f(y)|^{p}dy\right)^{1/p}
 \\ &\leq C\sum_{k=0}^{\infty}\frac{(2^{k}R)^{\alpha-Q}}
 {(1+2^{k}R)^{\gamma}}(2^{k}R)^{Q/p^{'}}
\left(\int_{2^{k}R\leq|xy^{-1}|<2^{k+1}R}|f(y)|^{p}dy\right)^{1/p}.
\end{align*}
This implies that
$$|I_{2}(x)|\leq C\|f\|_{L^{p, \phi}(\mathbb{G})}\sum_{k=0}^{\infty}\frac{(2^{k}R)^{\alpha-Q+Q/p_{1}}}{(1+2^{k}R)^{\gamma}}
\phi(2^{k}R)(2^{k}R)^{Q/p_{1}^{'}}.$$
Since $\phi(r)\leq Cr^{\beta}$, we write
$$|I_{2}(x)|\leq C\|f\|_{L^{p, \phi}(\mathbb{G})}\sum_{k=0}^{\infty}\frac{(2^{k}R)^{\alpha-Q+Q/p_{1}}}{(1+2^{k}R)^{\gamma}}
(2^{k}R)^{\beta+Q/p_{1}^{'}}.$$
Applying H\"{o}lder inequality again, we get
$$|I_{2}(x)|\leq C\|f\|_{L^{p, \phi}(\mathbb{G})}\left(\sum_{k=0}^{\infty}
\frac{(2^{k}R)^{(\alpha-Q)p_{1}+Q}}{(1+2^{k}R)^{\gamma p_{1}}}\right)^{1/p_{1}}\left(\sum_{k=0}^{\infty}(2^{k}R)^{\beta p_{1}^{'}+Q}\right)^{1/p_{1}^{'}}.$$
From the conditions $p_{1}<\frac{Q}{Q-\alpha}$ and $\beta<-\alpha$, we have $\beta p_{1}^{'}+Q<0$. By Theorem \ref{Pre1}, we also have
$$\left(\sum_{k=0}^{\infty}\frac{(2^{k}R)^{(\alpha-Q)p_{1}+Q}}{(1+2^{k}R)^{\gamma p_{1}}}\right)^{1/p_{1}}
\leq \left(\sum_{k\in\mathbb{Z}}\frac{(2^{k}R)^{(\alpha-Q)p_{1}+Q}}{(1+2^{k}R)^{\gamma p_{1}}}\right)^{1/p_{1}}\sim\|K_{\alpha,\gamma}\|_{L^{p_{1}}(\mathbb{G})}.$$
Using these, we arrive at
\begin{equation}\label{GM2}|I_{2}(x)|\leq C\|K_{\alpha,\gamma}\|_{L^{p_{1}}(\mathbb{G})}
\|f\|_{L^{p, \phi}(\mathbb{G})}R^{Q/p_{1}^{'}+\beta}.
\end{equation}
Summing up the estimates \eqref{GM1} and \eqref{GM2}, we obtain
\begin{align*}|I_{\alpha,\gamma}f(x)|\leq C\|K_{\alpha,\gamma}\|_{L^{p_{1}}(\mathbb{G})}(Mf(x)R^{Q/p_{1}^{'}}+\|f\|_{L^{p, \phi}(\mathbb{G})}R^{Q/p_{1}^{'}+\beta}).\end{align*}
Assuming that $f$ is not identically $0$ and that $Mf$ is finite everywhere, we can choose $R>0$ such that $R^{\beta}=\frac{Mf(x)}{\|f\|_{L^{p, \phi}}(\mathbb{G})}$, that is,
$$|I_{\alpha,\gamma}f(x)|\leq C\|K_{\alpha,\gamma}\|_{L^{p_{1}}(\mathbb{G})}\|f\|_{L^{p, \phi}(\mathbb{G})}^{-\frac{Q}{\beta p_{1}^{'}}}(Mf(x))^{1+\frac{Q}{\beta p_{1}^{'}}},$$
 for every $x\in \mathbb{G}$.
Setting $q=\frac{\beta p_{1}^{'} p}{\beta p_{1}^{'}+Q}$, for any $r>0$ we get
$$\left(\int_{|x|<r}|I_{\alpha,\gamma}f(x)|^{q}dx\right)^{\frac{1}{q}}\leq
 C\|K_{\alpha,\gamma}\|_{L^{p_{1}}(\mathbb{G})}\|f\|_{L^{p, \phi}(\mathbb{G})}^{1-p/q}\left(\int_{|x|<r}|Mf(x)|^{p}dx\right)^{1/q}.$$
Then we divide both sides by $\phi(r)^{p/q}r^{Q/q}$ to get
$$\frac{\left(\int_{|x|<r}|I_{\alpha,\gamma}f(x)|^{q}dx\right)^{\frac{1}{q}}}
{\psi(r)r^{Q/q}}\leq  C\|K_{\alpha,\gamma}\|_{L^{p_{1}}(\mathbb{G})}\|f\|_{L^{p, \phi}(\mathbb{G})}^{1-p/q}
\frac{\left(\int_{|x|<r}|Mf(x)|^{p}dx\right)^{1/q}}{\phi(r)^{p/q}r^{Q/q}},$$
where $\psi(r)=\phi(r)^{p/q}$. Now by taking the supremum over $r>0$, we obtain that
$$\|I_{\alpha,\gamma}f\|_{L^{q, \psi}(\mathbb{G})}\leq  C\|K_{\alpha,\gamma}\|_{L^{p_{1}}(\mathbb{G})}\|f\|_{L^{p, \phi}(\mathbb{G})}^{1-p/q}
\|Mf\|_{L^{p, \phi}(\mathbb{G})}^{p/q},$$
which gives \eqref{GMth1}, after applying estimate \eqref{Nakai}.
\end{proof}

Lemma \ref{incul_Mor_lem} gives the property that the Bessel-Riesz kernel belongs to Morrey spaces, which will be used in the next theorem.
\begin{thm}\label{BRth2}
Let $\mathbb{G}$ be a homogeneous group
of homogeneous dimension $Q$. Let $|\cdot|$ be a homogeneous quasi-norm. Let $\gamma>0$ and $0<\alpha<Q$. If $\phi(r)\leq Cr^{\beta}$ for every $r>0, \beta<-\alpha, $ $\frac{Q}{Q+\gamma-\alpha}<p_{2}\leq p_{1}<\frac{Q}{Q-\alpha}$ and $p_{2}\geq1$, then for all $f\in L^{p, \phi}(\mathbb{G})$ we have
\begin{equation}\label{GMth2}\|I_{\alpha,\gamma}f\|_{L^{q, \psi}(\mathbb{G})}\leq C_{p, \phi, Q}\|K_{\alpha,\gamma}\|_{L^{p_{2},p_{1}}(\mathbb{G})}\|f\|_{L^{p, \phi}(\mathbb{G})},
\end{equation}
where $1<p<\infty, q=\frac{\beta p_{1}^{'} p}{\beta p_{1}^{'}+Q}, \psi(r)=\phi(r)^{p/q}$.
\end{thm}
\begin{proof}[Proof of Theorem \ref{BRth2}] Similarly to the proof of Theorem \ref{BRth1}, we write $I_{\alpha,\gamma}f(x)$ in the form
$$I_{\alpha,\gamma}f(x):=I_{1}(x)+I_{2}(x),$$
where $I_{1}(x):=\int_{B(x,R)}\frac{|xy^{-1}|^{\alpha-Q}f(y)}{(1+|xy^{-1}|)^{\gamma}}dy$ and $I_{2}(x):=\int_{B^{c}(x,R)}\frac{|xy^{-1}|^{\alpha-Q}f(y)}{(1+|xy^{-1}|)^{\gamma}}dy$, $R>0$.

As before, we estimate the first term $I_{1}$ using the dyadic decomposition:
\begin{align*}|I_{1}(x)|&\leq\sum_{k=-\infty}^{-1}\int_{2^{k}R\leq |xy^{-1}|<2^{k+1}R}\frac{|xy^{-1}|^{\alpha-Q}|f(y)|}{(1+|xy^{-1}|)^{\gamma}}dy\\&
\leq \sum_{k=-\infty}^{-1}\frac{(2^{k}R)^{\alpha-Q}}{(1+2^{k}R)^{\gamma}}\int_{2^{k}R\leq |xy^{-1}|<2^{k+1}R}
|f(y)|dy\\&
\leq CMf(x)\sum_{k=-\infty}^{-1}\frac{(2^{k}R)^{\alpha-Q+Q/p_{2}}(2^{k}R)^{Q/p_{2}^{'}}}{(1+2^{k}R)^{\gamma}},
\end{align*} where $1\leq p_{2}\leq p_{1}$.
From this using the H\"{o}lder inequality for $\frac{1}{p_{2}}+\frac{1}{p_{2}^{'}}=1$, we get
$$|I_{1}(x)|\leq CMf(x)\left(\sum_{k=-\infty}^{-1}\frac{(2^{k}R)^{(\alpha-Q)p_{2}+Q}}{(1+2^{k}R)^{\gamma p_{2}}}\right)^{1/p_{2}}\left(\sum_{k=-\infty}^{-1}(2^{k}R)^{Q}\right)^{1/p_{2}^{'}}.$$
By virtue of \eqref{sum_equiv1}, we have
\begin{equation}\label{GM3}|I_{1}(x)|\leq C_{2}Mf(x)\left(\int_{0<|x|<R}K_{\alpha,\gamma}^{p_{2}}(x)dx\right)^{\frac{1}{p_{2}}}
R^{Q/p_{2}^{'}}\leq C
\|K_{\alpha,\gamma}\|_{L^{p_{2},p_{1}}(\mathbb{G})}Mf(x)R^{Q/p_{1}^{'}}.\end{equation}
Now for $I_{2}$ by using H\"{o}lder inequality for $\frac{1}{p}+\frac{1}{p^{'}}=1$, we have
$$|I_{2}(x)|\leq \sum_{k=0}^{\infty}\frac{(2^{k}R)^{\alpha-Q}}{(1+2^{k}R)^{\gamma}}(2^{k}R)^{Q/p^{'}}
\left(\int_{2^{k}R\leq|xy^{-1}|<2^{k+1}R}|f(y)|^{p}dy\right)^{1/p},$$
that is,
\begin{align*}|I_{2}(x)|&\leq C\|f\|_{L^{p, \phi}(\mathbb{G})}\sum_{k=0}^{\infty}\frac{(2^{k}R)^{\alpha}\phi(2^{k}R)}{(1+2^{k}R)^{\gamma}}
\frac{\left(\int_{2^{k}R\leq|xy^{-1}|<2^{k+1}R}dy\right)^{1/p_{2}}}{(2^{k}R)^{Q/p_{2}}}\\&
\leq C\|f\|_{L^{p, \phi}(\mathbb{G})}\sum_{k=0}^{\infty}\phi(2^{k}R)(2^{k}R)^{Q/p_{1}^{'}}\frac{\left(\int_{2^{k}R\leq|xy^{-1}|<2^{k+1}R}
K_{\alpha,\gamma}^{p_{2}}(xy^{-1})dy\right)^{1/p_{2}}}{(2^{k}R)^{Q/p_{2}-Q/p_{1}}},
\end{align*}
where we have used the following inequality
$$\left(\int_{2^{k}R\leq|xy^{-1}|<2^{k+1}R}
K_{\alpha,\gamma}^{p_{2}}(xy^{-1})dy\right)^{1/p_{2}}$$
\begin{equation}\label{sum_equiv2}
\sim\frac{(2^{k}R)^{(\alpha-Q)+Q/p_{2}}}{(1+2^{k}R)^{\gamma}}
\geq C\frac{(2^{k}R)^{(\alpha-Q)}}{(1+2^{k}R)^{\gamma}}
\left(\int_{2^{k}R\leq|xy^{-1}|<2^{k+1}R}dy\right)^{1/p_{2}}.
\end{equation}
Since we have $\phi(r)\leq Cr^{\beta}$ and
$$\frac{\left(\int_{2^{k}R\leq|xy^{-1}|<2^{k+1}R}
K_{\alpha,\gamma}^{p_{2}}(xy^{-1})dy\right)^{1/p_{2}}}{(2^{k}R)^{Q/p_{2}-Q/p_{1}}}\lesssim\|K_{\alpha,\gamma}\|_{L^{p_{2},p_{1}}(\mathbb{G})}$$
for every $k=0,1,2,...,$ we get
$$|I_{2}(x)|\leq C\|K_{\alpha,\gamma}\|_{L^{p_{2},p_{1}}(\mathbb{G})}\|f\|_{L^{p, \phi}(\mathbb{G})}\sum_{k=0}^{\infty}
(2^{k}R)^{\beta+Q/p_{1}^{'}}.$$
Taking into account $\beta+Q/p_{1}^{'}<0$, we have
\begin{equation}\label{GM4}|I_{2}(x)|\leq C\|K_{\alpha,\gamma}\|_{L^{p_{2},p_{1}}(\mathbb{G})}\|f\|_{L^{p, \phi}(\mathbb{G})}R^{\beta+Q/p_{1}^{'}}.\end{equation}
Summing up the estimates \eqref{GM3} and \eqref{GM4}, we obtain
\begin{align*}|I_{\alpha,\gamma}f(x)|\leq C\|K_{\alpha,\gamma}\|_{L^{p_{2},p_{1}}(\mathbb{G})}(Mf(x)R^{Q/p_{1}^{'}}+\|f\|_{L^{p, \phi}(\mathbb{G})}R^{\beta+Q/p_{1}^{'}}).\end{align*}
Assuming that $f$ is not identically $0$ and that $Mf$ is finite everywhere, we can choose $R>0$ such that $R^{\beta}=\frac{Mf(x)}{\|f\|_{L^{p, \phi}(\mathbb{G})}}$, which yields
$$
|I_{\alpha,\gamma}f(x)|\leq C\|K_{\alpha,\gamma}\|_{L^{p_{2},p_{1}}(\mathbb{G})}\|f\|_{L^{p, \phi}(\mathbb{G})}^{-\frac{Q}{\beta p_{1}^{'}}}(Mf(x))^{1+\frac{Q}{\beta p_{1}^{'}}}.
$$
Now by putting $q=\frac{\beta p_{1}^{'} p}{\beta p_{1}^{'}+Q}$, for any $r>0$ we obtain
$$
\left(\int_{|x|<r}|I_{\alpha,\gamma}f(x)|^{q}dx\right)^{\frac{1}{q}}\leq
C\|K_{\alpha,\gamma}\|_{L^{p_{2},p_{1}}(\mathbb{G})}\|f\|_{L^{p, \phi}(\mathbb{G})}^{1-p/q}\left(\int_{|x|<r}|Mf(x)|^{p}dx\right)^{1/q}.
$$
Then we divide both sides by $\phi(r)^{p/q}r^{Q/q}$ to get
$$
\frac{\left(\int_{|x|<r}|I_{\alpha,\gamma}f(x)|^{q}dx\right)^{\frac{1}{q}}}
{\psi(r)r^{Q/q}}\leq C\|K_{\alpha,\gamma}\|_{L^{p_{2},p_{1}}(\mathbb{G})}\|f\|_{L^{p, \phi}(\mathbb{G})}^{1-p/q}
\frac{\left(\int_{|x|<r}|Mf(x)|^{p}dx\right)^{1/q}}{\phi(r)^{p/q}r^{Q/q}},
$$
where $\psi(r)=\phi(r)^{p/q}$. Taking the supremum over $r>0$ and then using \eqref{Nakai}, we obtain the following desired result
\begin{multline*}
\|I_{\alpha,\gamma}f\|_{L^{q, \psi}(\mathbb{G})}\leq C\|K_{\alpha,\gamma}\|_{L^{p_{2},p_{1}}(\mathbb{G})}\|f\|_{L^{p, \phi}(\mathbb{G})}^{1-p/q}
\|Mf\|_{L^{p, \phi}(\mathbb{G})}^{p/q}\\
\leq C_{p, \phi, Q}\|K_{\alpha,\gamma}\|_{L^{p_{2},p_{1}}(\mathbb{G})}\|f\|_{L^{p, \phi}(\mathbb{G})},
\end{multline*}
completing the proof.
\end{proof}
By Lemma \ref{incul_Mor_lem}, we note that Theorem \ref{BRth2} implies Theorem \ref{BRth1}
$$\|I_{\alpha,\gamma}f\|_{L^{q, \psi}(\mathbb{G})}\leq C\|K_{\alpha,\gamma}\|_{L^{p_{2},p_{1}}(\mathbb{G})}\|f\|_{L^{p, \phi}(\mathbb{G})}
\leq C\|K_{\alpha,\gamma}\|_{L^{p_{1}}(\mathbb{G})}\|f\|_{L^{p, \phi}(\mathbb{G})}.$$
In order to improve our results, we present the following lemma, which states that the kernel $K_{\alpha,\gamma}$ belongs to the generalised Morrey space $L^{p_{2},\omega}(\mathbb{G})$ for some $p_{2}\geq1$ and some function $\omega$.

\begin{lem}\label{BRlem1}
Let $\mathbb{G}$ be a homogeneous group
of homogeneous dimension $Q$. Let $\gamma>0$, $p_{2}\geq1$ and $0<\alpha<Q$. If $\omega:\mathbb{R}^{+}\rightarrow\mathbb{R}^{+}$ satisfies
\begin{equation}\label{GM5}\int_{0<r\leq R}r^{(\alpha-Q)p_{2}+Q-1}dr\leq C\omega^{p_{2}}(R)R^{Q}\end{equation}
for every $R>0$, then $K_{\alpha,\gamma}\in L^{p_{2},\omega}(\mathbb{G})$.
\end{lem}
\begin{proof}[Proof of Lemma \ref{BRlem1}] Here, it is sufficient to evaluate the following integral around zero
$$\int_{|x|\leq R}K_{\alpha,\gamma}^{p_{2}}(x)dx=\int_{|x|\leq R}\frac{|x|^{(\alpha-Q)p_{2}}}{(1+|x|)^{\gamma p_{2}}}dx\leq
|\sigma|\int_{0<r\leq R}r^{(\alpha-Q)p_{2}+Q-1}dr\leq C\omega^{p_{2}}(R)R^{Q}.$$
By dividing both sides of this inequality by $\omega^{p_{2}}(R)R^{Q}$ and taking $p_{2}^{th}$-root, we obtain
$$\frac{\left(\int_{|x|\leq R}K_{\alpha,\gamma}^{p_{2}}(x)dx\right)^{1/p_{2}}}{\omega(R)R^{Q/p_{2}}}\leq C^{1/p_{2}}.$$
Then, we take the supremum over $R>0$ to get
$$\sup_{R>0}\frac{\left(\int_{|x|\leq R}K_{\alpha,\gamma}^{p_{2}}(x)dx\right)^{1/p_{2}}}{\omega(R)R^{Q/p_{2}}}<\infty,$$
which implies $K_{\alpha,\gamma}\in L^{p_{2},\omega}(\mathbb{G})$.
\end{proof}

\begin{thm}\label{BRth3}
Let $\mathbb{G}$ be a homogeneous group
of homogeneous dimension $Q$. Let $|\cdot|$ be a homogeneous quasi-norm. Let $\omega:\mathbb{R}^{+}\rightarrow\mathbb{R}^{+}$ satisfy the doubling condition and assume that $\omega(r)\leq Cr^{-\alpha}$ for every $r>0$, so that $K_{\alpha,\gamma}\in L^{p_{2}, \omega}(\mathbb{G})$ for $\frac{Q}{Q+\gamma-\alpha}<p_{2}<\frac{Q}{Q-\alpha}$ and $p_{2}\geq1$, where $0<\alpha<Q$ and $\gamma>0$. If $\phi(r)\leq Cr^{\beta}$ for every $r>0$, where $\beta<-\alpha<-Q-\beta$, then for all $f\in L^{p, \phi}(\mathbb{G})$ we have
\begin{equation}\label{GMth3}\|I_{\alpha,\gamma}f\|_{L^{q, \psi}(\mathbb{G})}\leq C_{p, \phi, Q}\|K_{\alpha,\gamma}\|_{L^{p_{2},\omega}(\mathbb{G})}\|f\|_{L^{p, \phi}(\mathbb{G})},
\end{equation}
where $1< p<\infty, q=\frac{\beta p}{\beta+Q-\alpha}, \psi(r)=\phi(r)^{p/q}$.
\end{thm}
\begin{proof}[Proof of Theorem \ref{BRth3}] As in the proof of Theorem \ref{BRth1}, we write
$$I_{\alpha,\gamma}f(x):=I_{1}(x)+I_{2}(x),$$
where $I_{1}(x):=\int_{B(x,R)}\frac{|xy^{-1}|^{\alpha-Q}f(y)}{(1+|xy^{-1}|)^{\gamma}}dy$ and $I_{2}(x):=\int_{B^{c}(x,R)}\frac{|xy^{-1}|^{\alpha-Q}f(y)}{(1+|xy^{-1}|)^{\gamma}}dy$, $R>0$.

First, we estimate $I_{1}$ by using the dyadic decomposition
\begin{align*}|I_{1}(x)|&\leq\sum_{k=-\infty}^{-1}\int_{2^{k}R\leq |xy^{-1}|<2^{k+1}R}\frac{|xy^{-1}|^{\alpha-Q}|f(y)|}{(1+|xy^{-1}|)^{\gamma}}dy\\&
\leq \sum_{k=-\infty}^{-1}\frac{(2^{k}R)^{\alpha-Q}}{(1+2^{k}R)^{\gamma}}\int_{2^{k}R\leq |xy^{-1}|<2^{k+1}R}
|f(y)|dy\\&
\leq CMf(x)\sum_{k=-\infty}^{-1}\frac{(2^{k}R)^{\alpha-Q+Q/p_{2}}(2^{k}R)^{Q/p_{2}^{'}}}{(1+2^{k}R)^{\gamma}}.
\end{align*}
From this using the H\"{o}lder inequality for $\frac{1}{p_{2}}+\frac{1}{p_{2}^{'}}=1$, we get
$$|I_{1}(x)|\leq CMf(x)\left(\sum_{k=-\infty}^{-1}\frac{(2^{k}R)^{(\alpha-Q)p_{2}+Q}}{(1+2^{k}R)^{\gamma p_{2}}}\right)^{1/p_{2}}\left(\sum_{k=-\infty}^{-1}(2^{k}R)^{Q}\right)^{1/p_{2}^{'}}.$$
By \eqref{sum_equiv1} we have
\begin{align*}|I_{1}(x)|&\leq CMf(x)\left(\int_{0<|x|<R}K_{\alpha,\gamma}^{p_{2}}(x)dx\right)^{\frac{1}{p_{2}}}
R^{Q/p_{2}^{'}}\\&\leq C\|K_{\alpha,\gamma}\|_{L^{p_{2},\omega}(\mathbb{G})}Mf(x)\omega(R)R^{Q},\end{align*}
and using $\omega(r)\leq Cr^{-\alpha}$, we arrive at
\begin{equation}\label{GM6}
|I_{1}(x)|\leq C
\|K_{\alpha,\gamma}\|_{L^{p_{2},\omega}(\mathbb{G})}Mf(x)R^{Q-\alpha}.\end{equation}
Now let us estimate the second term $I_{2}$:
\begin{align*}
|I_{2}(x)|&\leq \sum_{k=0}^{\infty}\frac{(2^{k}R)^{\alpha-Q}}{(1+2^{k}R)^{\gamma}}
\int_{2^{k}R\leq|xy^{-1}|<2^{k+1}R}|f(y)|dy\\&
\leq C\sum_{k=0}^{\infty}\frac{(2^{k}R)^{\alpha-Q}}{(1+2^{k}R)^{\gamma}}(2^{k}R)^{Q/p^{'}}
\left(\int_{2^{k}R\leq|xy^{-1}|<2^{k+1}R}|f(y)|^{p}dy\right)^{1/p}\\&
\leq C\|f\|_{L^{p, \phi}(\mathbb{G})}\sum_{k=0}^{\infty}
\frac{(2^{k}R)^{\alpha}\phi(2^{k}R)}{(1+2^{k}R)^{\gamma}}
\frac{\left(\int_{2^{k}R\leq|xy^{-1}|<2^{k+1}R}
dy\right)^{1/p_{2}}}{(2^{k}R)^{Q/p_{2}}},
\end{align*}
where we have used that $\left(\int_{2^{k}R\leq|xy^{-1}|<2^{k+1}R}
dy\right)^{1/p_{2}}\sim (2^{k}R)^{Q/p_{2}}$.
Using \eqref{sum_equiv2} we obtain
\begin{align*}|I_{2}(x)|&\leq C\|f\|_{L^{p, \phi}(\mathbb{G})}\sum_{k=0}^{\infty}
\frac{(2^{k}R)^{\alpha}\phi(2^{k}R)}{(2^{k}R)^{\alpha-Q}}\frac{\left(\int_{2^{k}R\leq|xy^{-1}|<2^{k+1}R}
K_{\alpha,\gamma}^{p_{2}} (xy^{-1})dy\right)^{1/p_{2}}}{(2^{k}R)^{Q/p_{2}}}.\end{align*}
Taking into account that $\phi(r)\leq Cr^{\beta}$ and $\omega(r)\leq Cr^{-\alpha}$ for every $r>0$, we have
\begin{align*}|I_{2}(x)|&\leq C\|f\|_{L^{p, \phi}(\mathbb{G})}\sum_{k=0}^{\infty}
(2^{k}R)^{Q-\alpha+\beta}\frac{\left(\int_{2^{k}R\leq|xy^{-1}|<2^{k+1}R}
K_{\alpha,\gamma}^{p_{2}} (xy^{-1})dy\right)^{1/p_{2}}}{\omega(2^{k}R)(2^{k}R)^{Q/p_{2}}}.\end{align*}
Since we have
$$\frac{\left(\int_{2^{k}R\leq|xy^{-1}|<2^{k+1}R}
K_{\alpha,\gamma}^{p_{2}} (xy^{-1})dy\right)^{1/p_{2}}}{\omega(2^{k}R)(2^{k}R)^{Q/p_{2}}}
\lesssim \|K_{\alpha,\gamma}\|_{L^{p_{2},\omega}(\mathbb{G})}$$ for every $k=0,1,2,...,$ it follows that
$$|I_{2}(x)|\leq C\|K_{\alpha,\gamma}\|_{L^{p_{2},\omega}(\mathbb{G})}\|f\|_{L^{p, \phi}(\mathbb{G})}\sum_{k=0}^{\infty}
(2^{k}R)^{Q-\alpha+\beta},$$
and since $Q-\alpha+\beta<0$, it implies that
\begin{equation}\label{GM7}
|I_{2}(x)|\leq C\|K_{\alpha,\gamma}\|_{L^{p_{2},\omega}(\mathbb{G})}\|f\|_{L^{p, \phi}(\mathbb{G})}R^{Q-\alpha+\beta}.
\end{equation}
Summing up the estimates \eqref{GM6} and \eqref{GM7}, we have
\begin{align*}|I_{\alpha,\gamma}f(x)|&\leq C\|K_{\alpha,\gamma}\|_{L^{p_{2},\omega}(\mathbb{G})}(Mf(x)R^{Q-\alpha}+\|f\|_{L^{p, \phi}(\mathbb{G})}R^{Q-\alpha+\beta}).\end{align*}
Assuming that $f$ is not identically $0$ and that $Mf$ is finite everywhere, we can choose $R>0$ such that $R^{\beta}=\frac{Mf(x)}{\|f\|_{L^{p, \phi}(\mathbb{G})}}$, that is
$$|I_{\alpha,\gamma}f(x)|\leq C\|K_{\alpha,\gamma}\|_{L^{p_{2},\omega}(\mathbb{G})}\|f\|_{L^{p, \phi}(\mathbb{G})}^{(\alpha-Q)/\beta}(Mf(x))^{1+(Q-\alpha)/\beta}.$$
Now by putting $q=\frac{\beta p}{\beta+Q-\alpha}$, for any $r>0$ we get
$$\left(\int_{|x|<r}|I_{\alpha,\gamma}f(x)|^{q}dx\right)^{\frac{1}{q}}\leq
C\|K_{\alpha,\gamma}\|_{L^{p_{2},\omega}(\mathbb{G})}\|f\|_{L^{p, \phi}(\mathbb{G})}^{1-p/q}\left(\int_{|x|<r}|Mf(x)|^{p}dx\right)^{1/q}.$$
Then we divide both sides by $\phi(r)^{p/q}r^{Q/q}$ to get
$$\frac{\left(\int_{|x|<r}|I_{\alpha,\gamma}f(x)|^{q}dx\right)^{\frac{1}{q}}}
{\psi(r)r^{Q/q}}\leq C\|K_{\alpha,\gamma}\|_{L^{p_{2},\omega}(\mathbb{G})}\|f\|_{L^{p, \phi}(\mathbb{G})}^{1-p/q}
\frac{\left(\int_{|x|<r}|Mf(x)|^{p}dx\right)^{1/q}}{\phi(r)^{p/q}r^{Q/q}},$$
where $\psi(r)=\phi(r)^{p/q}$. Finally, taking the supremum over $r>0$ and using \eqref{Nakai}, we obtain the desired result
\begin{multline*}
\|I_{\alpha,\gamma}f\|_{L^{q, \psi}(\mathbb{G})}\leq C\|K_{\alpha,\gamma}\|_{L^{p_{2},\omega}(\mathbb{G})}\|f\|_{L^{p, \phi}(\mathbb{G})}^{1-p/q}
\|Mf\|_{L^{p, \phi}(\mathbb{G})}^{p/q}\\
\leq C_{p, \phi, Q}\|K_{\alpha,\gamma}\|_{L^{p_{2},\omega}(\mathbb{G})}\|f\|_{L^{p, \phi}(\mathbb{G})},
\end{multline*}
completing the proof.
\end{proof}

\begin{rem}
We note that Theorems \ref{BRth1}, \ref{BRth2}, and \ref{BRth3} imply the results on the boundedness of Bessel-Riesz operators in Morrey spaces on homogeneous groups.  As in the Abelian case \cite{IGE16}, our results ensure that $I_{\alpha,\gamma}:L^{p,\phi}(\mathbb{G})\rightarrow L^{q,\phi^{p/q}}(\mathbb{G})$ is bounded. 
Indeed, if $\omega:\mathbb{R}^{+}\rightarrow \mathbb{R}^{+}$ satisfies conditions of Lemma \ref{BRlem1}, for $p_{1}\in\left(\frac{Q}{Q+\gamma-\alpha},\frac{Q}{Q-\alpha}\right)$ we have the inequality $R^{-Q/p_{1}}\leq \omega(R)$ for every $R>0$, and Theorem \ref{BRth3} gives a better estimate than Theorem \ref{BRth2}. For example, if we take $\omega(R):=(1+R^{Q/q_{1}})R^{-Q/p_{1}}$ for some $q_{1}>p_{1}$, then $\|K_{\alpha,\gamma}\|_{L^{p_{2},\omega}(\mathbb{G})}\leq\|K_{\alpha,\gamma}\|_{L^{p_{2},p_{1}}(\mathbb{G})}$. By Theorem \ref{BRth3} and Lemma \ref{incul_Mor_lem} we obtain
\begin{align*}
\|I_{\alpha,\gamma}f\|_{L^{q, \psi}(\mathbb{G})}&\leq C\|K_{\alpha,\gamma}\|_{L^{p_{2},\omega}(\mathbb{G})}\|f\|_{L^{p, \phi}(\mathbb{G})}\\
&\leq C\|K_{\alpha,\gamma}\|_{L^{p_{2},p_{1}}(\mathbb{G})}\|f\|_{L^{p, \phi}(\mathbb{G})}\\
&\leq C\|K_{\alpha,\gamma}\|_{L^{p_{1}}(\mathbb{G})}\|f\|_{L^{p, \phi}(\mathbb{G})}.
\end{align*}
Thus, we have shown that Theorem \ref{BRth3} gives the best estimate among the three. Moreover, it is shown that, in these estimates, the norm of Bessel-Riesz operators on generalised Morrey spaces is dominated by an appropriate norm of Bessel-Riesz kernels.
\end{rem}

\section{Inequalities for generalised Bessel-Riesz operator in generalised Morrey spaces}
\label{SEC:I_Gfracop}

In this section, we prove the boundedness of the generalised Bessel-Riesz operator $I_{\rho, \gamma}$ and establish Olsen type inequality for this operator in generalised Morrey spaces on homogeneous groups.

We define the generalised Bessel-Riesz operator $I_{\rho, \gamma}$ by
\begin{equation}\label{I_rho}
I_{\rho, \gamma}f(x):=\int_{\mathbb{G}}\frac{\rho(|xy^{-1}|)}{(1+|xy^{-1}|)^{\gamma}}f(y)dy,
\end{equation}
where $\gamma\geq0$, $\rho:\mathbb{R}^{+}\rightarrow\mathbb{R}^{+}$, $\rho$ satisfies the doubling condition \eqref{Pre02} and the following condition:
\begin{equation}\label{I_Pre01}
\int_{0}^{1}\frac{\rho(t)}{t^{\gamma-Q+1}}dt<\infty.
\end{equation} For $\rho(t)=t^{\alpha-Q}, 0<\alpha<Q$, we have the Bessel-Riesz kernel
$$I_{\rho, \gamma}=I_{\alpha, \gamma}=\frac{|xy^{-1}|^{\alpha-Q}}{(1+|xy^{-1}|)^{\gamma}}.$$

\begin{thm}\label{I_Gfracopthm}
Let $\mathbb{G}$ be a homogeneous group
of homogeneous dimension $Q$. Let $|\cdot|$ be a homogeneous quasi-norm and let $\gamma>0$. Let $\rho$ and $\phi$ satisfy the doubling condition \eqref{Pre02}. Let $\phi$ be surjective and for some $1<p<q<\infty$ satisfy
\begin{equation}\label{I_Gfracopthm1}\int_{r}^{\infty}\frac{\phi(t)^{p}}{t}dt\leq C_{1}\phi(r)^{p}, \end{equation}
and
\begin{equation}\label{I_Gfracopthm2}\phi(r)\int_{0}^{r}\frac{\rho(t)}{t^{\gamma-Q+1}}dt+\int_{r}^{\infty}
\frac{\rho(t)\phi(t)}{t^{\gamma-Q+1}}dt\leq C_{2} \phi(r)^{p/q}, 
\end{equation}
for all $r>0$. Then we have
\begin{equation}\label{I_Gfracopthm3}
\|I_{\rho, \gamma}f\|_{L^{q,\phi^{p/q}}(\mathbb{G})}\leq C_{p, q, \phi, Q}\|f\|_{L^{p,\phi}(\mathbb{G})}.
\end{equation}
\end{thm}

\begin{proof}[Proof of Theorem \ref{I_Gfracopthm}] For every $R>0$, let us write $I_{\rho, \gamma}f(x)$ in the form
$$I_{\rho, \gamma}f(x)=I_{1, \rho}(x)+I_{2, \rho}(x),$$
where $I_{1,\rho}(x):=\int_{B(x,R)}\frac{\rho(|xy^{-1}|)}{(1+|xy^{-1}|)^{\gamma}}f(y)dy$ and $I_{2,\rho}(x):=\int_{B^{c}(x,R)}\frac{\rho(|xy^{-1}|)}{(1+|xy^{-1}|)^{\gamma}}f(y)dy$.
For $I_{1,\rho}(x)$, we have
\begin{align*}
|I_{1,\rho}(x)|&\leq\int_{|xy^{-1}|<R}\frac{\rho(|xy^{-1}|)}{(1+|xy^{-1}|)^{\gamma}}|f(y)|dy
\leq\int_{|xy^{-1}|<R}\frac{\rho(|xy^{-1}|)}{|xy^{-1}|^{\gamma}}|f(y)|dy
\\ &= \sum_{k=-\infty}^{-1}\int_{2^{k}R\leq|xy^{-1}|<2^{k+1}R}\frac{\rho(|xy^{-1}|)}{|xy^{-1}|^{\gamma}}|f(y)|dy.
\end{align*}
By virtue of \eqref{Pre02}, we get
\begin{align*}
|I_{1,\rho}(x)|&\leq C \sum_{k=-\infty}^{-1}\frac{\rho(2^{k}R)}{(2^{k}R)^{\gamma}}\int_{|xy^{-1}|<2^{k+1}R}|f(y)|dy
\\ & \leq C Mf(x) \sum_{k=-\infty}^{-1}\frac{\rho(2^{k}R)}{(2^{k}R)^{\gamma-Q}}
\\ & \leq C Mf(x) \sum_{k=-\infty}^{-1}\int_{2^{k}R}^{2^{k+1}R}\frac{\rho(t)}{t^{\gamma-Q+1}}dt
\\ & = C Mf(x) \int_{0}^{R}\frac{\rho(t)}{t^{\gamma-Q+1}}dt,
\end{align*}
where we have used the fact that
\begin{equation}\label{doubling2}\int_{2^{k}R}^{2^{k+1}R}\frac{\rho(t)}{t^{\gamma-Q+1}}dt\geq C\frac{\rho(2^{k}R)}{(2^{k}R)^{\gamma-Q+1}}2^{k}R\geq C\frac{\rho(2^{k}R)}{(2^{k}R)^{\gamma-Q}}.
\end{equation}
Now, using \eqref{I_Gfracopthm2}, we obtain
\begin{equation}\label{I_Gfracopt1}
|I_{1,\rho}(x)| \leq C Mf(x) \phi(R)^{(p-q)/q}.
\end{equation}
For $I_{2,\rho}(x)$, we have
\begin{align*}
|I_{2,\rho}(x)|&\leq\int_{|xy^{-1}|\geq R}\frac{\rho(|xy^{-1}|)}{(1+|xy^{-1}|)^{\gamma}}|f(y)|dy\leq\int_{|xy^{-1}|\geq R}\frac{\rho(|xy^{-1}|)}{|xy^{-1}|^{\gamma}}|f(y)|dy
\\ &= \sum_{k=0}^{\infty}\int_{2^{k}R\leq|xy^{-1}|<2^{k+1}R}\frac{\rho(|xy^{-1}|)}{|xy^{-1}|^{\gamma}}|f(y)|dy.
\end{align*}
Applying \eqref{Pre02}, we get
$$|I_{2,\rho}(x)|\leq C\sum_{k=0}^{\infty}\frac{\rho(2^{k}R)}{(2^{k}R)^{\gamma}}\int_{|xy^{-1}|<2^{k+1}R}|f(y)|dy.$$
From this using the H\"{o}lder inequality, we obtain
\begin{align*}|I_{2,\rho}(x)|&\leq C \sum_{k=0}^{\infty}\frac{\rho(2^{k}R)}{(2^{k}R)^{\gamma}}
\left(\int_{|xy^{-1}|<2^{k+1}R}dy\right)^{1-\frac{1}{p}}
\left(\int_{|xy^{-1}|<2^{k+1}R}|f(y)|dy\right)^{\frac{1}{p}}\\&
\leq C\sum_{k=0}^{\infty}\frac{\rho(2^{k}R)}{(2^{k}R)^{\gamma-Q+\frac{Q}{p}}}
\left(\int_{|xy^{-1}|<2^{k+1}R}|f(y)|dy\right)^{\frac{1}{p}} \\ &
\leq C \|f\|_{L^{p,\phi}(\mathbb{G})} \sum_{k=0}^{\infty}\frac{\rho(2^{k+1}R)\phi(2^{k+1}R)}{(2^{k}R)^{\gamma-Q}}
\\ & \leq C  \|f\|_{L^{p,\phi}(\mathbb{G})} \sum_{k=0}^{\infty}\int_{2^{k}R}^{2^{k+1}R}\frac{\rho(t)\phi(t)}{t^{\gamma-Q+1}}dt
\\ & = C \|f\|_{L^{p,\phi}(\mathbb{G})} \int_{R}^{\infty}\frac{\rho(t)\phi(t)}{t^{\gamma-Q+1}}dt,
\end{align*}
where we have used the fact that
$$\int_{2^{k}R}^{2^{k+1}R}\frac{\rho(t)\phi(t)}{t^{\gamma-Q+1}}dt\geq C\frac{\rho(2^{k+1}R)\phi(2^{k+1}R)}{(2^{k+1}R)^{\gamma-Q+1}}2^{k}R\geq C\frac{\rho(2^{k+1}R)\phi(2^{k+1}R)}{(2^{k}R)^{\gamma-Q}}.$$
Now, using \eqref{I_Gfracopthm2}, we obtain
\begin{equation}\label{I_Gfracopt2} |I_{2,\rho}(x)|\leq C \|f\|_{L^{p,\phi}(\mathbb{G})}\phi(R)^{p/q}.
\end{equation}
Summing the two estimates \eqref{I_Gfracopt1} and \eqref{I_Gfracopt2}, we arrive at
$$|I_{\rho, \gamma}f(x)|\leq C
(Mf(x)\phi(R)^{(p-q)/q}+\|f\|_{L^{p,\phi}(\mathbb{G})}\phi(R)^{p/q}).$$
Assuming that $f$ is not identically $0$ and that $Mf$ is finite everywhere and then using the fact that $\phi$ is surjective, we can choose $R>0$ such that $\phi(R)=Mf(x)\cdot\|f\|_{L^{p,\phi}(\mathbb{G})}^{-1}$. Thus, for every $x\in \mathbb{G}$, we have
$$|I_{\rho, \gamma}f(x)|\leq C
Mf(x)^{\frac{p}{q}}\|f\|_{L^{p,\phi}(\mathbb{G})}^{\frac{q-p}{q}}.$$
It follows that
$$\left(\int_{B(0,r)}|I_{\rho, \gamma}f(x)|^{q}\right)^{1/q}\leq C
\left(\int_{B(0,r)}|Mf(x)|^p\right)^{1/q}\|f\|_{L^{p,\phi}(\mathbb{G})}^{\frac{q-p}{q}},$$
then we divide both sides by $\phi(r)^{p/q}r^{Q/q}$ to get
\begin{align*}\frac{1}{\phi(r)^{p/q}}\left(\frac{1}{r^{Q}}\int_{B(0,r)}|I_{\rho, \gamma}f(x)|^{q}\right)^{1/q}\leq C
\frac{1}{\phi(r)^{p/q}}\left(\frac{1}{r^{Q}}\int_{B(0,r)}|Mf(x)|^p\right)^{1/q}\|f\|_{L^{p,\phi}(\mathbb{G})}^{\frac{q-p}{q}}.
\end{align*}
Taking the supremum over $r>0$ and using the boundedness of the maximal operator $M$ on $L^{p,\phi}(\mathbb{G})$ from \eqref{Nakai}, we obtain
$$\|I_{\rho, \gamma}f\|_{L^{q,\phi^{p/q}}(\mathbb{G})}\leq C_{p, q, \phi, Q}\|f\|_{L^{p,\phi}(\mathbb{G})}.$$
This completes the proof.
\end{proof}

Now let show the Olsen type inequalities for the generalised Bessel-Riesz operator $I_{\rho, \gamma}$.
\begin{thm}\label{I_Olsenthm}
Let $\mathbb{G}$ be a homogeneous group
of homogeneous dimension $Q$. Let $|\cdot|$ be a homogeneous quasi-norm and let $\gamma>0$. Let $\rho$ and $\phi$ satisfy the doubling condition \eqref{Pre02}. Let $\phi$ be surjective and satisfy \eqref{I_Gfracopthm1}-\eqref{I_Gfracopthm2}. Then we have
\begin{equation}\label{I_Olsenthm1}\|W\cdot I_{\rho, \gamma}f\|_{L^{p,\phi}(\mathbb{G})}\leq C_{p, \phi, Q}\|W\|_{L^{p_{2},\phi^{p/p_{2}}}(\mathbb{G})}\|f\|_{L^{p,\phi}(\mathbb{G})}, \quad 1<p<p_{2}<\infty,
\end{equation}
provided that $W\in L^{p_{2},\phi^{p/p_{2}}}(\mathbb{G})$.
\end{thm}
\begin{proof}[Proof of Theorem \ref{I_Olsenthm}] By using H\"{o}lder inequality, we have
\begin{multline*}
\frac{1}{r^{Q}}\int_{B(0,r)}|W\cdot I_{\rho, \gamma}f(x)|^{p}dx \leq \left(\frac{1}{r^{Q}}\int_{B(0,r)}|W(x)|^{p_{2}}dx\right)^{p/p_{2}} \\
\left(\frac{1}{r^{Q}}\int_{B(0,r)}|I_{\rho, \gamma}f(x)|^{\frac{pp_{2}}{p_{2}-p}}dx\right)^{\frac{p_{2}-p}{p_{2}}}.
\end{multline*}
Now let us take the $p$-th roots and then divide both sides by $\phi(r)$ to obtain
\begin{align*}
\frac{1}{\phi(r)}\left(\frac{1}{r^{Q}}\int_{B(0,r)}|W\cdot I_{\rho, \gamma}f(x)|^{p}dx\right)^{1/p} &
\leq \frac{1}{\phi(r)^{p/p_{2}}}\left(\frac{1}{r^{Q}}\int_{B(0,r)}|W(x)|^{p_{2}}dx\right)^{1/p_{2}}\\&
\times \frac{1}{\phi(r)^{\frac{p_{2}-p}{p_{2}}}}\left(\frac{1}{r^{Q}}\int_{B(0,r)}|I_{\rho, \gamma}f(x)|^{\frac{pp_{2}}{p_{2}-p}}dx\right)^{\frac{p_{2}-p}{pp_{2}}}.
\end{align*}
By taking the supremum over $r>0$ and using the inequality \eqref{I_Gfracopthm3}, we get
$$\|W\cdot I_{\rho, \gamma}f\|_{L^{p,\phi}(\mathbb{G})}\leq C_{p, \phi, Q} \|W\|_{L^{p_{2},\phi^{p/p_{2}}}(\mathbb{G})}\|I_{\rho, \gamma}f\|_{L^{\frac{pp_{2}}{p_{2}-p},\phi^{\frac{p_{2}-p}{p_{2}}}}(\mathbb{G})}.$$
Taking into account that $1<p<\frac{pp_{2}}{p_{2}-p}< \infty$ and putting $q=\frac{pp_{2}}{p_{2}-p}$ in \eqref{I_Gfracopthm3}, we obtain \eqref{I_Olsenthm1}.
\end{proof}

\section{Generalised fractional integral operators in generalised Morrey spaces}
\label{SEC:Gfracop}

In this section, we prove the boundedness of the generalised fractional integral operators and establish Olsen type inequality in generalised Morrey spaces on homogeneous groups.

We define the generalised fractional integral operator $T_{\rho}$ by
\begin{equation}\label{T_rho}
T_{\rho}f(x):=\int_{\mathbb{G}}\frac{\rho(|xy^{-1}|)}{|xy^{-1}|^{Q}}f(y)dy,
\end{equation}
where $\rho:\mathbb{R}^{+}\rightarrow\mathbb{R}^{+}$ satisfies the doubling condition \eqref{Pre02} and the condition
\begin{equation}\label{Pre01}
\int_{0}^{1}\frac{\rho(t)}{t}dt<\infty.
\end{equation} 
As in the Abelian case, for $\rho(t)=t^{\alpha}$, $0<\alpha<Q$, we have the Riesz transform
$$T_{\rho}f(x)=I_{\alpha}f(x)=\int_{\mathbb{G}}\frac{1}{|xy^{-1}|^{Q-\alpha}}f(y)dy.$$

\begin{thm}\label{Gfracopthm}
Let $\mathbb{G}$ be a homogeneous group
of homogeneous dimension $Q$. Let $|\cdot|$ be a homogeneous quasi-norm. Let $\rho$ and $\phi$ satisfy the doubling condition \eqref{Pre02}. Let $\phi$ be also surjective and satisfy, for some $1<p<q<\infty$, the inequalities
\begin{equation}\label{Gfracopthm1}
\int_{r}^{\infty}\frac{\phi(t)^{p}}{t}dt\leq C_{1}\phi(r)^{p}, \end{equation}
and
\begin{equation}\label{Gfracopthm2}\phi(r)\int_{0}^{r}\frac{\rho(t)}{t}dt+\int_{r}^{\infty}\frac{\rho(t)\phi(t)}{t}dt\leq C_{2} \phi(r)^{p/q},
\end{equation}
for all $r>0$. Then we have
\begin{equation}\label{Gfracopthm3}
\|T_{\rho}f\|_{L^{q,\phi^{p/q}}(\mathbb{G})}\leq C_{p, q, \phi, Q}\|f\|_{L^{p,\phi}(\mathbb{G})}.
\end{equation}
\end{thm}

\begin{proof}[Proof of Theorem \ref{Gfracopthm}] 
For every $R>0$, let us write $T_{\rho}f(x)$ in the form
$$T_{\rho}f(x)=T_{1}(x)+T_{2}(x),$$
where $T_{1}(x):=\int_{B(x,R)}\frac{\rho(|xy^{-1}|)}{(|xy^{-1}|)^{Q}}f(y)dy$ and $T_{2}(x):=\int_{B^{c}(x,R)}\frac{\rho(|xy^{-1}|)}{(|xy^{-1}|)^{Q}}f(y)dy$.
For $T_{1}(x)$, we have
\begin{align*}
|T_{1}(x)|&\leq\int_{|xy^{-1}|<R}\frac{\rho(|xy^{-1}|)}{|xy^{-1}|^{Q}}|f(y)|dy
\\ &= \sum_{k=-\infty}^{-1}\int_{2^{k}R\leq|xy^{-1}|<2^{k+1}R}\frac{\rho(|xy^{-1}|)}{|xy^{-1}|^{Q}}|f(y)|dy.
\end{align*}
By view of \eqref{Pre02}, we get
\begin{align*}
|T_{1}(x)|&\leq C \sum_{k=-\infty}^{-1}\frac{\rho(2^{k}R)}{(2^{k}R)^{Q}}\int_{|xy^{-1}|<2^{k+1}R}|f(y)|dy
\\ & \leq C Mf(x) \sum_{k=-\infty}^{-1}\rho(2^{k}R)
\\ & \leq C Mf(x) \sum_{k=-\infty}^{-1}\int_{2^{k}R}^{2^{k+1}R}\frac{\rho(t)}{t}dt
\\ & = C Mf(x) \int_{0}^{R}\frac{\rho(t)}{t}dt.
\end{align*}
Here we have used the fact that
\begin{equation}\label{doubling2}\int_{2^{k}R}^{2^{k+1}R}\frac{\rho(t)}{t}dt\geq C\rho(2^{k}R)\int_{2^{k}R}^{2^{k+1}R}\frac{1}{t}dt= C\rho(2^{k}R)\ln2.
\end{equation}
Now, using \eqref{Gfracopthm2}, we obtain
\begin{equation}\label{Gfracopt1}
|T_{1}(x)| \leq C Mf(x) \phi(R)^{(p-q)/q}.
\end{equation}
For $T_{2}(x)$, we have
\begin{align*}
|T_{2}(x)|&\leq\int_{|xy^{-1}|\geq R}\frac{\rho(|xy^{-1}|)}{|xy^{-1}|^{Q}}|f(y)|dy
\\ &= \sum_{k=0}^{\infty}\int_{2^{k}R\leq|xy^{-1}|<2^{k+1}R}\frac{\rho(|xy^{-1}|)}{|xy^{-1}|^{Q}}|f(y)|dy.
\end{align*}
Applying \eqref{Pre02}, we get
$$|T_{2}(x)|\leq C\sum_{k=0}^{\infty}\frac{\rho(2^{k}R)}{(2^{k}R)^{Q}}\int_{|xy^{-1}|<2^{k+1}R}|f(y)|dy.$$
From this using the H\"{o}lder inequality, we obtain
\begin{align*}|T_{2}(x)|&\leq C \sum_{k=0}^{\infty}\frac{\rho(2^{k}R)}{(2^{k}R)^{Q}}
\left(\int_{|xy^{-1}|<2^{k+1}R}dy\right)^{1-1/p}
\left(\int_{|xy^{-1}|<2^{k+1}R}|f(y)|dy\right)^{1/p}\\&
\leq C\sum_{k=0}^{\infty}\frac{\rho(2^{k}R)}{(2^{k}R)^{Q/p}}
\left(\int_{|xy^{-1}|<2^{k+1}R}|f(y)|dy\right)^{1/p} \\ &
\leq C \|f\|_{L^{p,\phi}(\mathbb{G})} \sum_{k=0}^{\infty}\rho(2^{k+1}R)\phi(2^{k+1}R)
\\ & \leq C  \|f\|_{L^{p,\phi}(\mathbb{G})} \sum_{k=0}^{\infty}\int_{2^{k}R}^{2^{k+1}R}\frac{\rho(t)\phi(t)}{t}
\\ & = C \|f\|_{L^{p,\phi}(\mathbb{G})} \int_{R}^{\infty}\frac{\rho(t)\phi(t)}{t},
\end{align*}
where we have used the fact that
$$\int_{2^{k}R}^{2^{k+1}R}\frac{\rho(t)\phi(t)}{t}dt\geq C\rho(2^{k+1}R)\phi(2^{k+1}R)\int_{2^{k}R}^{2^{k+1}R}\frac{1}{t}dt= C\rho(2^{k+1}R)\phi(2^{k+1}R)\ln2.$$
Now, in view of \eqref{Gfracopthm2}, we obtain
\begin{equation}\label{Gfracopt2} |T_{2}(x)|\leq C \|f\|_{L^{p,\phi}(\mathbb{G})}\phi(R)^{p/q}.
\end{equation}
Summing the two estimates \eqref{Gfracopt1} and \eqref{Gfracopt2}, we arrive at
$$|T_{\rho}f(x)|\leq C
(Mf(x)\phi(R)^{(p-q)/q}+\|f\|_{L^{p,\phi}(\mathbb{G})}\phi(R)^{p/q}).$$
Assuming that $f$ is not identically $0$ and that $Mf$ is finite everywhere and then using the fact that $\phi$ is surjective, we can choose $R>0$ such that $\phi(R)=Mf(x)\cdot\|f\|_{L^{p,\phi}(\mathbb{G})}^{-1}$. Thus, for every $x\in \mathbb{G}$, we have
$$|T_{\rho}f(x)|\leq C
Mf(x)^{\frac{p}{q}}\|f\|_{L^{p,\phi}(\mathbb{G})}^{\frac{q-p}{q}}.$$
It follows that
$$\left(\int_{B(0,r)}|T_{\rho}f(x)|^{q}\right)^{1/q}\leq C
\left(\int_{B(0,r)}|Mf(x)|^p\right)^{1/q}\|f\|_{L^{p,\phi}(\mathbb{G})}^{\frac{q-p}{q}},$$
then we divide both sides by $\phi(r)^{p/q}r^{Q/q}$ to get
\begin{align*}\frac{1}{\phi(r)^{p/q}}\left(\frac{1}{r^{Q}}\int_{B(0,r)}|T_{\rho}f(x)|^{q}\right)^{1/q}\leq C
\frac{1}{\phi(r)^{p/q}}\left(\frac{1}{r^{Q}}\int_{B(0,r)}|Mf(x)|^p\right)^{1/q}\|f\|_{L^{p,\phi}(\mathbb{G})}^{\frac{q-p}{q}}.
\end{align*}
Taking the supremum over $r>0$ and using the boundedness of the maximal operator $M$ on $L^{p,\phi}(\mathbb{G})$ \eqref{Nakai}, we obtain
$$\|T_{\rho}f\|_{L^{q,\phi^{p/q}}(\mathbb{G})}\leq C_{p, q, \phi, Q}\|f\|_{L^{p,\phi}(\mathbb{G})}.$$
The proof is complete.
\end{proof}

Now let us turn to the Olsen type inequalities for the generalised fractional integral operator $T_{\rho}$ and Bessel-Riesz operator $I_{\alpha,\gamma}$.
\begin{thm}\label{Olsenthm}
Let $\mathbb{G}$ be a homogeneous group
of homogeneous dimension $Q$. Let $|\cdot|$ be a homogeneous quasi-norm. Let $\rho$ and $\phi$ satisfy the doubling condition \eqref{Pre02}. Let $\phi$ be also surjective and satisfy \eqref{Gfracopthm1}-\eqref{Gfracopthm2}. Then we have
\begin{equation}\label{Olsenthm1}\|W\cdot T_{\rho}f\|_{L^{p,\phi}(\mathbb{G})}\leq C_{p, \phi, Q}\|W\|_{L^{p_{2},\phi^{p/p_{2}}}(\mathbb{G})}\|f\|_{L^{p,\phi}(\mathbb{G})}, \quad 1<p<p_{2}<\infty,
\end{equation}
provided that $W\in L^{p_{2},\phi^{p/p_{2}}}(\mathbb{G})$.
\end{thm}
\begin{proof}[Proof of Theorem \ref{Olsenthm}] By using H\"{o}lder inequality, we have
\begin{multline*}
\frac{1}{r^{Q}}\int_{B(0,r)}|W\cdot T_{\rho}f(x)|^{p}dx \\
\leq \left(\frac{1}{r^{Q}}\int_{B(0,r)}|W(x)|^{p_{2}}dx\right)^{p/p_{2}}
\left(\frac{1}{r^{Q}}\int_{B(0,r)}|T_{\rho}f(x)|^{\frac{pp_{2}}{p_{2}-p}}dx\right)^{\frac{p_{2}-p}{p_{2}}}.
\end{multline*}
Now let us take the $p$-th roots and then divide both sides by $\phi(r)$ to obtain
\begin{align*}
\frac{1}{\phi(r)}\left(\frac{1}{r^{Q}}\int_{B(0,r)}|W\cdot T_{\rho}f(x)|^{p}dx\right)^{1/p} &
\leq \frac{1}{\phi(r)^{p/p_{2}}}\left(\frac{1}{r^{Q}}\int_{B(0,r)}|W(x)|^{p_{2}}dx\right)^{1/p_{2}}\\&
\times \frac{1}{\phi(r)^{\frac{p_{2}-p}{p_{2}}}}\left(\frac{1}{r^{Q}}\int_{B(0,r)}|T_{\rho}f(x)|^{\frac{pp_{2}}{p_{2}-p}}dx\right)^{\frac{p_{2}-p}{pp_{2}}}.
\end{align*}
By taking the supremum over $r>0$ and using the inequality \eqref{Gfracopthm3}, we get
$$\|W\cdot T_{\rho}f\|_{L^{p,\phi}(\mathbb{G})}\leq C_{p, \phi, Q} \|W\|_{L^{p_{2},\phi^{p/p_{2}}}(\mathbb{G})}\|T_{\rho}f\|_{L^{\frac{pp_{2}}{p_{2}-p},\phi^{\frac{p_{2}-p}{p_{2}}}}(\mathbb{G})}.$$

Taking into account that $1<p<\frac{pp_{2}}{p_{2}-p}< \infty$ and putting $q=\frac{pp_{2}}{p_{2}-p}$ in \eqref{Gfracopthm3}, we obtain \eqref{Olsenthm1}.
\end{proof}

\begin{thm}\label{Olsenthm2} Let $\mathbb{G}$ be a homogeneous group
of homogeneous dimension $Q$. Let $|\cdot|$ be a homogeneous quasi-norm. Let $\omega:\mathbb{R}^{+}\rightarrow\mathbb{R}^{+}$ satisfy the doubling condition and assume that $\omega(r)\leq Cr^{-\alpha}$ for every $r>0$, so that $K_{\alpha,\gamma}\in L^{p_{2}, \omega}(\mathbb{G})$ for $\frac{Q}{Q+\gamma-\alpha}<p_{2}<\frac{Q}{Q-\alpha}$ and $p_{2}\geq1$, where $0<\alpha<Q$, $1<p<\infty, q=\frac{\beta p}{\beta+Q-\alpha}$ and $\gamma>0$. If $\phi(r)\leq Cr^{\beta}$ for every $r>0$, where $\beta<-\alpha<-Q-\beta$, then we have
\begin{equation}\label{Olsenthm2_1}\|W\cdot I_{\alpha,\gamma}f\|_{L^{p,\phi}(\mathbb{G})}\leq C_{p, \phi, Q}\|W\|_{L^{p_{2},\phi^{p/p_{2}}}(\mathbb{G})}\|f\|_{L^{p,\phi}(\mathbb{G})},
\end{equation}
provided that $W\in L^{p_{2},\phi^{p/p_{2}}}(\mathbb{G})$, where $\frac{1}{p_{2}}=\frac{1}{p}-\frac{1}{q}$.
\end{thm}
\begin{proof}[Proof of Theorem \ref{Olsenthm2}] As in Theorem \ref{Olsenthm}, by using H\"{o}lder inequality for $\frac{p}{p_{2}}+\frac{p}{q}=1$, we have
$$\frac{1}{r^{Q}}\int_{B(0,r)}|W\cdot I_{\alpha,\gamma}f(x)|^{p}dx \leq \left(\frac{1}{r^{Q}}\int_{B(0,r)}|W(x)|^{p_{2}}dx\right)^{p/p_{2}}
\left(\frac{1}{r^{Q}}\int_{B(0,r)}|I_{\alpha,\gamma}f(x)|^{q}dx\right)^{p/q}.$$
Now we take the $p$-th roots and then divide both sides by $\phi(r)$ to get
\begin{align*}
\frac{1}{\phi(r)}\left(\frac{1}{r^{Q}}\int_{B(0,r)}|W\cdot I_{\alpha,\gamma}f(x)|^{p}dx\right)^{1/p} &
\leq \frac{1}{\phi(r)^{p/p_{2}}}\left(\frac{1}{r^{Q}}\int_{B(0,r)}|W(x)|^{p_{2}}dx\right)^{1/p_{2}}\\&
\times \frac{1}{\phi(r)^{p/q}}\left(\frac{1}{r^{Q}}\int_{B(0,r)}|I_{\alpha,\gamma}f(x)|^{q}dx\right)^{1/q}.
\end{align*}
By taking the supremum over $r>0$, we have
$$\|W\cdot I_{\alpha,\gamma}f\|_{L^{p,\phi}(\mathbb{G})}\leq C \|W\|_{L^{p_{2},\phi^{p/p_{2}}}(\mathbb{G})}\|I_{\alpha,\gamma}f\|_{L^{q,\phi^{p/q}}(\mathbb{G})},$$
which implies \eqref{Olsenthm2_1} in view of Theorem \ref{BRth3} after putting $\psi(r)=\phi(r)^{p/q}$.
\end{proof}

\section{Inequalities for the modified version of generalised fractional integral operator in Campanato spaces}
\label{SEC:Campanato}

In this section, we prove the boundedness of the modified version of the operator $T_{\rho}$ in Campanato spaces on homogeneous groups.

We define the generalised Campanato space by
\begin{equation}\label{GCampanato}\mathcal{L}^{p,\phi}(\mathbb{G}):=\{f\in L^{p}_{loc}(\mathbb{G}):\|f\|_{\mathcal{L}^{p,\phi}(\mathbb{G})}<\infty\},
\end{equation}
where
$$\|f\|_{\mathcal{L}^{p,\phi}(\mathbb{G})}:=\sup_{r>0}\frac{1}{\phi(r)}\left(\frac{1}{r^{Q}}\int_{B(0,r)}|f(x)-f_{B}|^{p}dx\right)^{1/p},$$
with $f_B=f_{B(0,r)}:=\frac{1}{r^{Q}}\int_{B(0,r)}f(y)dy$, and we assume that $\frac{\phi(r)}{r}$ is nonincreasing.

Next, for the function $\rho:\mathbb{R}^{+}\rightarrow\mathbb{R}^{+}$, we define the modified version of the generalised fractional integral operator $T_{\rho}$ by
\begin{equation}\label{ModT}\widetilde{T_{\rho}}f(x):=\int_{\mathbb{G}}\left(\frac{\rho(|xy^{-1}|)}{|xy^{-1}|^{Q}}
-\frac{\rho(|y|)(1-\chi_{B(0,1)}(y))}{|y|^{Q}}\right)f(y)dy,
\end{equation}
where $B(0,1):=\{x\in \mathbb{G}: |x|<1\}$ and $\chi_{B(0,1)}$ is the characteristic function of $B(0,1)$. In this definition, we assume that $\rho$ satisfies \eqref{Pre01}, \eqref{Pre02} and the following conditions:
\begin{equation}\label{Pre03}
\int_{r}^{\infty}\frac{\rho(t)}{t^{2}}dt\leq C_{1}\frac{\rho(r)}{r} \;\;{\rm for \;\;all}\;\; r>0;
\end{equation}
\begin{equation}\label{Pre04}
\frac{1}{2}\leq \frac{r}{s}\leq 2 \Rightarrow \left|\frac{\rho(r)}{r^{Q}}-\frac{\rho(s)}{s^{Q}}\right|\leq C_{2}|r-s|\frac{\rho(s)}{s^{Q+1}}.
\end{equation}
For instance, the function $\rho(r)=r^{\alpha}$ satisfies \eqref{Pre01}, \eqref{Pre02} and \eqref{Pre04} for $0<\alpha<Q$, and also satisfies \eqref{Pre03} for $0<\alpha<1$.
\begin{thm}\label{Campanatothm}
Let $\mathbb{G}$ be a homogeneous group
of homogeneous dimension $Q$. Let $|\cdot|$ be a homogeneous quasi-norm. Let $\rho$ satisfy \eqref{Pre01}, \eqref{Pre02}, \eqref{Pre03}, \eqref{Pre04}, and let $\phi$ satisfy the doubling condition \eqref{Pre02} and $\int_{1}^{\infty}\frac{\phi(t)}{t}dt<\infty$. If
\begin{equation}\label{Campanatothm1}\int_{r}^{\infty}\frac{\phi(t)}{t}dt\int_{0}^{r}\frac{\rho(t)}{t}dt+r\int_{r}^{\infty}\frac{\rho(t)\phi(t)}{t^{2}}dt\leq C_{3} \psi(r) \;{\it for \;\;all} \;r>0,
\end{equation}
then we have
\begin{equation}\label{Campanato3}
\|\widetilde{T}_{\rho}f\|_{\mathcal{L}^{p,\psi}(\mathbb{G})}\leq C_{p, \phi, Q}\|f\|_{\mathcal{L}^{p,\phi}(\mathbb{G})}, \;\;1<p<\infty.
\end{equation}
\end{thm}
\begin{proof}[Proof of Theorem \ref{Campanatothm}] For every $x\in B(0,r)$ and $f\in \mathcal{L}^{p,\phi}(\mathbb{G})$, let us write $\widetilde{T}_{\rho}f$ in the following form:
$$\widetilde{T}_{\rho}f(x)=\widetilde{T}_{B(0,r)}(x)+C^{1}_{B(0,r)}+C^{2}_{B(0,r)}=\widetilde{T}^{1}_{B(0,r)}(x)+\widetilde{T}^{2}_{B(0,r)}(x)
+C^{1}_{B(0,r)}+C^{2}_{B(0,r)},$$
where
$$\widetilde{T}_{B(0,r)}(x):=\int_{\mathbb{G}}(f(y)-f_{B(0,2r)})\left(\frac{\rho(|xy^{-1}|)}{|xy^{-1}|^{Q}}-
\frac{\rho(|y|)(1-\chi_{B(0,2r)}(y))}{|y|^{Q}}\right)dy,$$
$$C^{1}_{B(0,r)}:=\int_{\mathbb{G}}(f(y)-f_{B(0,2r)})\left(\frac{\rho(|y|)(1-\chi_{B(0,2r)}(y))}{|y|^{Q}}-
\frac{\rho(|y|)(1-\chi_{B(0,1)}(y))}{|y|^{Q}}\right)dy,$$
$$C^{2}_{B(0,r)}:=\int_{\mathbb{G}}f_{B(0,2r)}\left(\frac{\rho(|xy^{-1}|)}{|xy^{-1}|^{Q}}
-\frac{\rho(|y|)(1-\chi_{B(0,1)}(y))}{|y|^{Q}}\right)dy,$$
$$\widetilde{T}^{1}_{B(0,r)}(x):=\int_{B(0,2r)}(f(y)-f_{B(0,2r)})\frac{\rho(|xy^{-1}|)}{|xy^{-1}|^{Q}}dy,$$
$$\widetilde{T}^{2}_{B(0,r)}(x):=\int_{B^{c}(0,2r)}(f(y)-f_{B(0,2r)})\left(\frac{\rho(|xy^{-1}|)}{|xy^{-1}|^{Q}}-
\frac{\rho(|y|)}{|y|^{Q}}\right)dy.$$
Since
\begin{multline*}
\left|\frac{\rho(|y|)(1-\chi_{B(0,2r)}(y))}{|y|^{Q}}-
\frac{\rho(|y|)(1-\chi_{B(0,1)}(y))}{|y|^{Q}}\right|\;
\\ \leq
\begin{cases} 0, |y|< \min(1,2r) \;{\rm or}\; |y|\geq \max(1,2r);\\
\frac{\rho(|y|)}{|y|^{Q}}=const, \;{\rm otherwise}, \end{cases}
\end{multline*}
$C^{1}_{B(0,r)}$ is finite.

Now let us show that $C^{2}_{B(0,r)}$ is finite. For this it is enough to prove that the following integral is finite:
$$\int_{\mathbb{G}}\left(\frac{\rho(|xy^{-1}|)}{|xy^{-1}|^{Q}}
-\frac{\rho(|y|)(1-\chi_{B(0,1)}(y))}{|y|^{Q}}\right)dy$$
$$=\int_{\mathbb{G}}\left(\frac{\rho(|xy^{-1}|)}{|xy^{-1}|^{Q}}
-\frac{\rho(|y|)}{|y|^{Q}}\right)dy+\int_{B(0,1)}\frac{\rho(|y|)}{|y|^{Q}}dy.$$
Let us denote $A:=\int_{\mathbb{G}}\left(\frac{\rho(|xy^{-1}|)}{|xy^{-1}|^{Q}}
-\frac{\rho(|y|)}{|y|^{Q}}\right)dy$. For large $R>0$, we write $A$ in the form
$$A=A_{1}+A_{2}+A_{3},$$
where
$$A_{1}=\int_{B(x,R)}\frac{\rho(|xy^{-1}|)}{|xy^{-1}|^{Q}}dy-\int_{B(0,R)}\frac{\rho(|y|)}{|y|^{Q}}dy,$$
$$A_{2}=\int_{B(x,R+r)\backslash B(x,R)}\frac{\rho(|xy^{-1}|)}{|xy^{-1}|^{Q}}dy-\int_{B(x,R+r)\backslash B(0,R)}\frac{\rho(|y|)}{|y|^{Q}}dy,$$
$$A_{3}=\int_{B^{c}(x,R+r)}\left(\frac{\rho(|xy^{-1}|)}{|xy^{-1}|^{Q}}-\frac{\rho(|y|)}{|y|^{Q}}\right)dy.$$
Since we have $\int_{0}^{1}\frac{\rho(t)}{t}dt<+\infty$, it implies that
$$\frac{\rho(|xy^{-1}|)}{|xy^{-1}|^{Q}}, \frac{\rho(|y|)}{|y|^{Q}} \in L^{1}_{loc}(\mathbb{G}),$$
and hence $A_{1}=0.$ By \eqref{Pre04} we have
$$A_{3}\leq\int_{B^{c}(x,R+r)}\left|\frac{\rho(|xy^{-1}|)}{|xy^{-1}|^{Q}}-\frac{\rho(|y|)}{|y|^{Q}}\right|dy$$
$$\leq C\int_{B^{c}(x,R+r)}||xy^{-1}|-|y||\frac{\rho(|xy^{-1}|)}{|xy^{-1}|^{Q+1}}dy.$$
By using the triangle inequality (see e.g. \cite[Theorem 3.1.39, p.113]{FR}) and symmetric property of homogeneous quasi-norms, we get
$$A_{3}\leq C||x|+|y^{-1}|-|y||\int_{R+r}^{+\infty}\int_{\wp}\frac{\rho(t)}{t^{Q+1}}t^{Q-1}d\sigma(y)dt$$
$$\leq C|\sigma|r\int_{R+r}^{+\infty}\frac{\rho(t)}{t^{2}}dt.$$
The inequality \eqref{Pre03} implies that the last integral is integrable and $|A_{3}|\rightarrow0$ as $R\rightarrow +\infty$.
For $A_{2}$, we have
$$|A_{2}|\leq \int_{B(x,R+r)\backslash B(x,R-r)}\left(\frac{\rho(|xy^{-1}|)}{|xy^{-1}|^{Q}}+\frac{\rho(|y|)}{|y|^{Q}}\right)dy$$
$$\sim((R+r)^{Q}-(R-r)^{Q})\frac{\rho(R)}{R^{Q}}\leq C r \frac{\rho(R)}{R},$$
and taking into account the conditions \eqref{Pre02} and \eqref{Pre03}, we obtain
$$|A_{2}|\leq C r \frac{\rho(R)}{R}\rightarrow 0 \;\;{\rm as} \;\;R\rightarrow+\infty.$$
Since $A\rightarrow0$ as $R\rightarrow +\infty$, we have $A=0$ and hence
$$\int_{\mathbb{G}}\left(\frac{\rho(|xy^{-1}|)}{|xy^{-1}|^{Q}}
-\frac{\rho(|y|)(1-\chi_{B(0,1)}(y))}{|y|^{Q}}\right)dy=\int_{B(0,1)}\frac{\rho(|y|)}{|y|^{Q}}dy<\infty,$$
which implies that $C^{2}_{B(0,r)}$ is finite.

Now before estimating $\widetilde{T}^{1}_{B(0,r)}$, let us denote $\widetilde{f}:=(f-f_{B(0,2r)})\chi_{B(0,2r)}$ and $\widetilde{\phi}(r):=\int_{r}^{\infty}\frac{\phi(t)}{t}dt$. Then, we have
\begin{align*}
|\widetilde{T}^{1}_{B(0,r)}(x)|&\leq\int_{B(0,2r)}|\widetilde{f}(y)|\frac{\rho(|xy^{-1}|)}{|xy^{-1}|^{Q}}dy
\\ &= \sum_{k=-\infty}^{0}\int_{2^{k}r\leq|xy^{-1}|<2^{k+1}r}\frac{\rho(|xy^{-1}|)}{|xy^{-1}|^{Q}}|\widetilde{f}(y)|dy.
\end{align*}
By using \eqref{Pre02} and \eqref{doubling2}, we get
\begin{align*}
|\widetilde{T}^{1}_{B(0,r)}(x)|&\leq C \sum_{k=-\infty}^{0}\frac{\rho(2^{k}r)}{(2^{k}r)^{Q}}\int_{|xy^{-1}|<2^{k+1}r}|\widetilde{f}(y)|dy
\\ & \leq C M\widetilde{f}(x) \sum_{k=-\infty}^{0}\rho(2^{k}r)
\\ & \leq C M\widetilde{f}(x) \sum_{k=-\infty}^{0}\rho(2^{k-1}r)
\\ & \leq C M\widetilde{f}(x) \sum_{k=-\infty}^{0}\int_{2^{k-1}r}^{2^{k}r}\frac{\rho(t)}{t}dt
\\ & = C M\widetilde{f}(x) \int_{0}^{r}\frac{\rho(t)}{t}dt.
\end{align*}
Now using \eqref{Campanatothm1}, we have
$$|\widetilde{T}^{1}_{B(0,r)}(x)|\leq C\frac{\psi(r)}{\widetilde{\phi}(r)}M\widetilde{f}(x).$$
It follows that
$$\frac{1}{\psi(r)}\left(\frac{1}{r^{Q}}\int_{B(0,r)}|\widetilde{T}^{1}_{B(0,r)}(x)|^{p}dx\right)^{1/p}\leq
C\frac{1}{\widetilde{\phi}(r)r^{Q/p}}\left(\int_{B(0,r)}|M\widetilde{f}(x)|^{p}dx\right)^{1/p}$$
$$\leq
C\frac{1}{\widetilde{\phi}(r)r^{Q/p}}\|\widetilde{f}\|_{L^{p}(\mathbb{G})},$$
where we used \eqref{HLmax2}.

By Minkowski inequality, we have
$$\frac{1}{\widetilde{\phi}(r)r^{Q/p}}\|\widetilde{f}\|_{L^{p}(\mathbb{G})}=
\frac{1}{\widetilde{\phi}(r)r^{Q/p}}\|(f-f_{B(0,2r)})\chi_{B(0,2r)}\|_{L^{p}(\mathbb{G})}$$
$$\leq C\frac{1}{\widetilde{\phi}(r)r^{Q/p}}
(\|(f-\sigma(f))_{\chi_{B(0,2r)}}\|_{L^{p}(\mathbb{G})}+(2r)^{Q/p}|f_{B(0,2r)}-\sigma(f)|),$$
where $\sigma(f)=\underset{r\rightarrow \infty}{\rm lim}f_{B(0,r)}$.

We obtain the following inequalities exactly in the same way as in the Abelian case (see \cite{EGN04}, Section 6)
\begin{equation}\label{EGN1}
\|f-\sigma(f)\|_{L^{p,\widetilde{\phi}}(\mathbb{G})}\leq C_{1}\|f\|_{\mathcal{L}^{p,\phi}(\mathbb{G})},
\end{equation}
and
\begin{equation}\label{EGN2}
|f_{B(0,r)}-\sigma(f)|\leq C_{2}\|f\|_{\mathcal{L}^{p,\phi}(\mathbb{G})}\widetilde{\phi}(r).
\end{equation}
Finally, using these inequalities we get our estimate for $\widetilde{T}^{1}_{B(0,r)}$ as
\begin{equation}\label{Campanato1}|\widetilde{T}^{1}_{B(0,r)}(x)|\leq C\|f\|_{\mathcal{L}^{p,\phi}(\mathbb{G})}.
\end{equation}
Now let us estimate $\widetilde{T}^{2}_{B(0,r)}$. By \eqref{Pre02} and \eqref{Pre04}, we have
$$|\widetilde{T}^{2}_{B(0,r)}(x)|\leq \int_{B^{c}(0,2r)}|f(y)-f_{B(0,2r)}|\left|\frac{\rho(|xy^{-1}|)}{|xy^{-1}|^{Q}}-\frac{\rho(|y|)}{|y|^{Q}}\right|dy$$
$$\leq C ||xy^{-1}|-|y||\int_{|y|\geq 2r}|f(y)-f_{B(0,2r)}|\frac{\rho(|y|)}{|y|^{Q+1}}dy.$$
By using the triangle inequality (see e.g. \cite[Theorem 3.1.39, p.113]{FR}) and symmetric property of homogeneous quasi-norms, we get
$$|\widetilde{T}^{2}_{B(0,r)}(x)|\leq  C ||x|+|y^{-1}|-|y||\int_{|y|\geq 2r}|f(y)-f_{B(0,2r)}|\frac{\rho(|y|)}{|y|^{Q+1}}dy$$
$$\leq C |x|\int_{|y|\geq 2r}|f(y)-f_{B(0,2r)}|\frac{\rho(|y|)}{|y|^{Q+1}}dy$$
$$=C |x|\sum_{k=2}^{\infty}\int_{2^{k-1}r\leq|y|< 2^{k}r}\frac{\rho(|y|)|f(y)-f_{B(0,2r)}|}{|y|^{Q+1}}dy.$$
By using \eqref{Pre02} and H\"{o}lder inequality, we have
$$|\widetilde{T}^{2}_{B(0,r)}(x)|\leq C|x|\sum_{k=2}^{\infty}\frac{\rho(2^{k}r)}{(2^{k}r)^{Q+1}}\int_{|y|< 2^{k}r}|f(y)-f_{B(0,2r)}|dy$$
$$\leq C|x|\sum_{k=2}^{\infty}\frac{\rho(2^{k}r)}{2^{k}r}\left(\frac{1}{(2^{k}r)^{Q}}\int_{|y|< 2^{k}r}|f(y)-f_{B(0,2r)}|^{p}dy\right)^{1/p}.$$
As in the Abelian case (\cite{EGN04}), we have
$$\left(\frac{1}{(2^{k}r)^{Q}}\int_{B(0,2^{k}r)}|f(y)-f_{B(0,2r)}|^{p}dy\right)^{1/p}\leq C\|f\|_{\mathcal{L}^{p,\phi}(\mathbb{G})}\int_{2r}^{2^{k+1}r}\frac{\phi(s)}{s}ds,$$
for every $k\geq2$. The inequality \eqref{doubling2} implies that
$$\int^{2^{k+1}r}_{2^{k}r}\frac{\rho(t)}{t^{2}}dt\geq \frac{1}{2^{k+1}r}\int^{2^{k+1}r}_{2^{k}r}\frac{\rho(t)}{t}dt\geq C\frac{\rho(2^{k}r)}{2^{k}r}.$$
By using the last two inequalities, we get
$$|\widetilde{T}^{2}_{B(0,r)}(x)|\leq C |x| \|f\|_{\mathcal{L}^{p,\phi}(\mathbb{G})} \sum_{k=2}^{\infty}\frac{\rho(2^{k}r)}{2^{k}r}\int_{2r}^{2^{k+1}r}\frac{\phi(s)}{s}ds$$
$$\leq C |x| \|f\|_{\mathcal{L}^{p,\phi}(\mathbb{G})} \sum_{k=2}^{\infty}\int^{2^{k+1}r}_{2^{k}r}\frac{\rho(t)}{t^{2}}\left(\int_{2r}^{t}\frac{\phi(s)}{s}ds\right)dt $$
$$\leq C |x| \|f\|_{\mathcal{L}^{p,\phi}(\mathbb{G})} \int^{\infty}_{2r}\frac{\rho(t)}{t^{2}}
\left(\int_{2r}^{t}\frac{\phi(s)}{s}ds\right)dt$$
$$=C |x| \|f\|_{\mathcal{L}^{p,\phi}(\mathbb{G})} \int^{\infty}_{2r}
\left(\int_{s}^{\infty}\frac{\rho(t)}{t^{2}}dt\right)\frac{\phi(s)}{s}ds.$$
Using \eqref{Pre03} and then \eqref{Campanatothm1}, it implies that
$$|\widetilde{T}^{2}_{B(0,r)}(x)|\leq C r \|f\|_{\mathcal{L}^{p,\phi}(\mathbb{G})} \int^{\infty}_{2r}\frac{\rho(s)\phi(s)}{s^{2}}ds\leq C \psi(r) \|f\|_{\mathcal{L}^{p,\phi}(\mathbb{G})}.$$
It follows that
\begin{equation}\label{Campanato2}\frac{1}{\psi(r)}\left(\frac{1}{r^{Q}}\int_{B(0,r)}|\widetilde{T}^{2}_{B(0,r)}(x)|^{p}dx\right)^{1/p}\leq
C \|f\|_{\mathcal{L}^{p,\phi}(\mathbb{G})}.
\end{equation}
Summing the estimates \eqref{Campanato1} and \eqref{Campanato2}, we obtain \eqref{Campanato3}.
\end{proof}

\end{document}